\documentclass[1 [leqno,11pt]{amsart}
\usepackage{amssymb, amsmath}
\def\longformule#1#2{
\displaylines{ \qquad{#1} \hfill\cr \hfill {#2} \qquad\cr } }
\def\inte#1{
\displaystyle\mathop{#1\kern0pt}^\circ }

\let\pa=\partial

\let\d=\delta
\let\e=\varepsilon

\let\r=\rho

\let\f=\frac

\let\p=\psi
\let\om=\omega

\let\D=\Delta

\let\Om=\Omega
\let\wt=\widetilde

\def\cC{{\mathcal C}}

\def\cF{{\mathcal F}}
\def\cG{{\mathcal G}}

\def\cR{{\mathcal R}}

\def\pa{\partial}
\def\Ga{\Gamma}
\def\grad{\nabla}

\def\virgp{\raise 2pt\hbox{,}}
\def\cdotpv{\raise 2pt\hbox{;}}

\def\eqdefa{\buildrel\hbox{\footnotesize def}\over =}

\def\C{\mathop{\mathbb C\kern 0pt}\nolimits}
\def\DD{\mathop{\mathbb D\kern 0pt}\nolimits}
\def\EE{\mathop{{\mathbb E \kern 0pt}}\nolimits}
\def\K{\mathop{\mathbb K\kern 0pt}\nolimits}
\def\N{\mathop{\mathbb N\kern 0pt}\nolimits}
\def\Q{\mathop{\mathbb Q\kern 0pt}\nolimits}
\def\R{\mathop{\mathbb R\kern 0pt}\nolimits}
\def\SS{\mathop{\mathbb S\kern 0pt}\nolimits}
\def\ZZ{\mathop{\mathbb Z\kern 0pt}\nolimits}
\def\TT{\mathop{\mathbb T\kern 0pt}\nolimits}
\def\P{\mathop{\mathbb P\kern 0pt}\nolimits}

\def\dv{\mbox{div}}

\def\dive{\mathop{\rm div}\nolimits}
\def\curl{\mathop{\rm curl}\nolimits}

\def\no{\noindent}
\def\na{\nabla}
\def\p{\partial}
\def\th{\theta}

\newcommand{\w}[1]{\langle {#1} \rangle}
\newcommand{\beq}{\begin{equation}}
\newcommand{\eeq}{\end{equation}}
\newcommand{\ben}{\begin{eqnarray}}
\newcommand{\een}{\end{eqnarray}}
\newcommand{\beno}{\begin{eqnarray*}}
\newcommand{\eeno}{\end{eqnarray*}}
\newcommand{\andf}{\quad\hbox{and}\quad}
\newcommand{\with}{\quad\hbox{with}\quad}

\newtheorem{thm}{Theorem}[section]
\newtheorem{lem}{Lemma}[section]
\newtheorem{rmk}{Remark}[section]
\newtheorem{col}{Corollary}[section]
\newtheorem{prop}{Proposition}[section]
\renewcommand{\theequation}{\thesection.\arabic{equation}}

\begin{document}

\title[Axisymmetric solutions of 3-D Inhomogeneous Navier-Stokes system]
{Global smooth axisymmetric solutions of 3-D Inhomogenenous
incompressible Navier-Stokes system}

\author[H.  Abidi]{Hammadi Abidi}
\address[H.  Abidi]{D\'epartement de Math\'ematiques
Facult\'e des Sciences de Tunis Campus universitaire 2092 Tunis,
Tunisia} \email{habidi@univ-evry.fr}
\author[P. Zhang]{Ping Zhang} \address[P. Zhang]{Academy of Mathematics $\&$ Systems
Science\\
and  Hua Loo-Keng Key Laboratory of Mathematics, Chinese Academy of
Sciences,  Beijing 100190, CHINA.} \email{zp@amss.ac.cn}

\date{Sept.  10, 2014}

\maketitle

\begin{abstract} In this paper, we investigate the global regularity
to 3-D  inhomogeneous incompressible Navier-Stokes system with
axisymmetric initial data which does not have swirl component for
the initial velocity. We first prove that the $L^\infty$ norm to the
quotient of the inhomogeneity by $r,$ namely
$a/r\eqdefa\bigl(1/\r-1\bigr)\bigl/r,$ controls the regularity of
the solutions. Then  we  prove the global regularity of such
solutions provided that the $L^\infty$ norm of $a_0/r$ is
sufficiently small. Finally, with additional assumption that the
initial velocity belongs to $L^p$ for some $p\in [1,2),$ we  prove
that the velocity field decays to zero with exactly the same rate as
the classical Navier-Stokes system.
\end{abstract}


\vskip 0.2cm


\renewcommand{\theequation}{\thesection.\arabic{equation}}
\setcounter{equation}{0}
\section{Introduction}

In this paper, we consider the global existence of smooth
 solutions to the following 3-D inhomogeneous
incompressible Navier-Stokes equations with axisymmetric initial
data which does not have swirl component for the initial velocity:
\begin{equation}\label{eq1.1}
\quad\left\{\begin{array}{l}
\displaystyle \pa_t \rho + \dv ( \rho u)=0,\qquad (t,x)\in \R^+\times\R^3,\\
\displaystyle \pa_t (\rho u) + \dv(\rho u \otimes u) -\Delta u+ \grad\Pi=0, \\
\displaystyle \dv u = 0, \\
\displaystyle (\rho, u)|_{t=0}=(\rho_0, u_{0}).
\end{array}\right.
\end{equation}
where $\rho, u=(u^1,u^2, u^z)$ stand for the density and  velocity
of the fluid respectively, and $\Pi$  is a scalar pressure function.
Such system describes a  fluid that is  incompressible but has
non-constant density. Basic examples are mixture of incompressible
and non reactant flows, flows with complex structure (e.g. blood
flow or model of rivers), fluids containing a melted substance, etc.
\smallskip

 A lot of recent works have been dedicated to the mathematical study
of the above system. Global weak solutions with finite energy have
been constructed by  Simon in \cite{Simon} (see also the book by
Lions \cite{LP} for the variable viscosity case). In the case of
smooth data with no vacuum, the existence of strong unique solutions
goes back to the work of Ladyzhenskaya and  Solonnikov in \cite{LS}.
 More
precisely, they  considered the system \eqref{eq1.1} in a bounded
domain $\Om$ with homogeneous Dirichlet boundary condition for $u.$
Under the assumption that $u_0\in W^{2-\frac2p,p}(\Om)$ $(p>d)$ is
divergence free and vanishes on  $\p\Om$ and that $\r_0\in C^1(\Om)$
is bounded away from zero, then they \cite{LS} proved
\begin{itemize}
\item Global well-posedness in dimension $d=2;$
\item Local well-posedness in dimension $d=3.$ If in addition $u_0$ is small in $W^{2-\frac2p,p}(\Om),$
then global well-posedness holds true.
\end{itemize}
Lately,  Danchin and Mucha  \cite{dm} established the well-posedness
of \eqref{eq1.1} in the whole space $\R^d$ in the so-called
\emph{critical functional framework} for small perturbations of some
positive constant density. The basic idea are to use functional
spaces (or norms) that  is \emph{scaling invariant} under the
following transformation:
\begin{equation}\label{eq:scalinginhomo}
(\rho,u,\Pi)(t,x)\longmapsto (\rho,\lambda u,\lambda^2\Pi)
(\lambda^2 t,\lambda x),\qquad (\rho_0,u_0)(x)\longmapsto
(\rho_0,\lambda u_0)(\lambda x).
\end{equation} One may check \cite{AGZ2,dz} and the references therein for the recent
progresses along this line.

\smallskip

On the other hand, we recall that except the initial data have some
special structure, it is still not known whether or not the System
\eqref{eq1.1} has a unique global smooth solution with large smooth
initial data, even for the classical Navier-Stokes system $(NS),$
which corresponds to $\r=1$ in \eqref{eq1.1}.  For instance,
Ukhovskii and Yudovich \cite{UY}, and independently Ladyzhenskaya
\cite{La} proved the global existence of generalized solution along
with its uniqueness and regularity  for $(NS)$ with initial data
which is axisymmetric and without swirl. Leonardi, M\'{a}lek,
Ne\u{c}as and Pokorny \cite{LMNP} gave a refined proof of the same
result in \cite{ La, UY}.  The first author \cite{Abidi} improved
the regularity of the initial data to be $u_0\in H^{\f12}.$ In
general, the global wellposedness of $(NS)$ with axisymmetric
initial data is still open (see \cite{CL, ZZ} for instance).

\smallskip

 Let $x=(x_1, x_2, z)\in
\mathbb{R}^3,$ we  denote the cylindrical coordinates of $x$ by $(r,
\theta, z),$ i. e., $ r(x_1, x_2)\eqdefa\sqrt{x_1^2+x_2^2}, \quad
\theta(x_1, x_2) \eqdefa \tan^{-1} \frac{x_2}{x_1}$ with $r \in [0,
\infty), \, \theta \in [0, 2\pi]$ and $z \in \mathbb{R},$ and
\begin{equation*}
e_{r}\eqdefa(\cos \theta, \sin \theta, 0), \quad
e_{\theta}\eqdefa(-\sin \theta, \cos \theta, 0),\quad
e_{z}\eqdefa(0, 0, 1).
\end{equation*} We are concerned  here with the global existence of
axisymmetric smooth solutions to \eqref{eq1.1} which does not have
the swirl component for the velocity field. This means solution of
the form:
\begin{equation}\label{eq1.2}
\begin{split}
 &\rho(t,x_1, x_2, z)=\rho(t, r, z), \quad \Pi(t, x_1, x_2, z)=\Pi(t, r, z), \\
&  u(t,x_1,x_2,z)=u^r(t,r, z)e_r+u^{z}(t,r, z)e_{z}.
\end{split}
\end{equation}

By virtue of \eqref{eq1.1} and \eqref{eq1.2}, we find that
$(\r,u,\Pi)$ verifies
\begin{equation}\label{eq1.3}
\begin{cases}
&\pa_t \rho+u^r \pa_r \rho+u^{z} \pa_z \rho=0,\\
 &\r\pa_t u^r + \r u^r \pa_r u^r+\r u^{z} \pa_z u^r+\pa_r \Pi
-  \bigl(\frac{1}{r}\pa_r(r\pa_r
u^r)+\pa_z^2u^r -\frac{u^r}{r^2}\bigr)=0,\\
 &\r\pa_t
u^{z} + \r u^r \pa_r u^{z}+\r u^{z} \pa_z u^{z} +\pa_z
 \Pi-  \bigl(\frac{1}{r}\pa_r(r\pa_r
u^z)+\pa_z^2u^z \bigr)=0,\\
&\pa_ru^r+\frac{u^r}{r}+\pa_zu^z=0,\\
&\r|_{t=0}=\r_0\andf (u^r,u^z)|_{t=0}=(u_0^r,u_0^z).
\end{cases}
\end{equation}

\no {\bf Equation of  vorticity $\omega \eqdefa \pa_zu^r-\pa_r
u^z$:} we get, by taking $\pa_z (1.4)_2-\pa_r (1.4)_3,$ that
\begin{equation}\label{eq1.6}
\pa_t \omega +  u^r \pa_r \omega +u^{z} \pa_z \omega-\frac{1}{r}u^r
\omega+\pa_z\Bigl(\frac{\pa_r
\Pi}{\rho}\Bigr)-\pa_r\Bigl(\frac{\pa_z \Pi}{\rho}\Bigr)
-\pa_z\Bigl(\frac{\pa_z \omega}{\rho}\Bigr)-\pa_r\Bigl(\frac{\pa_r
\omega+\omega/r}{\rho}\Bigr)=0.
\end{equation}

\no {\bf Equation of  $\Ga\eqdefa \frac{\omega}{r}$:} in view of
\eqref{eq1.6}, one has
\begin{equation}\label{eq1.7}
\pa_t \Ga +  u^r \pa_r \Ga +u^{z} \pa_z
\Ga+\frac{1}{r}\pa_z\Bigl(\frac{\pa_r
 \Pi}{\rho}\Bigr)-\frac{1}{r}\pa_r\Bigl(\frac{\pa_z  \Pi}{\rho}\Bigr) -
\pa_z\Bigl(\frac{\pa_z
\Ga}{\rho}\Bigr)-\frac{1}{r}\pa_r\Bigl(\frac{r\pa_r
\Ga+2\Ga}{\rho}\Bigr)=0.
\end{equation}

As for the classical Navier-Stokes system $(NS)$ in \cite{La, UY},
the quantity $\Ga$ will play a  crucial role  to prove the global
well-poseness of \eqref{eq1.3}. The main result of this paper states
as follows:

\begin{thm}\label{thm1.1}
{\sl   Let $a_0\eqdefa \frac{1}{\rho_0}-1\in L^{2}\cap L^\infty $
with  $\frac{a_0}{r}\in L^\infty,$ and there exist positive
constants $m, M$ so that \beq\label{eq1.11} 0<m\leq \r_0\leq M. \eeq
Let $u_0=u_0^re_r+u_0^ze_z\in H^1$ be a solenoidal vector filed with
$\f{u_0^r}r$ and $\Ga_0\eqdefa\frac{\omega_0}{r}$ belonging to $
L^2.$ Then

\begin{itemize}

\item[(1)]
there exists a positive time $T^\ast$ so that \eqref{eq1.3} has a
unique solution $(\r, u)$ on $[0,T^\ast)$ which satisfies for any
$T<T^\ast$ \beq\label{eq1.13}\begin{split} &\r\in
L^\infty((0,T)\times\R^3),\quad u\in
\cC([0,T];H^1(\R^3))\with \na u \in L^2((0,T);H^1(\R^3))\\
&\sup_{t\in (0,T]}\Bigl(t\w{t}\bigl(\|u_t(t)\|_{L^2}^2+\|
u(t)\|_{\dot{H}^2}^2+\|\na\Pi(t)\|_{L^2}^2\bigr)+\int_0^tt'\w{t'}\|\na
u_t(t')\|_{L^2}^2\,dt'\Bigr)<\infty.
\end{split}
\eeq If $T^\ast<\infty,$ there holds
 \beq\label{eq4.6} \lim_{t\to
T^\ast}\bigl\|\f{a(t)}r\bigr\|_{L^\infty}=\infty. \eeq

\item[(2)]
If we assume  moreover that
 \beq\label{eq1.8}
\bigl\|\f{a_0}r\bigr\|_{L^\infty}\leq \e_0 \eeq  for some
sufficiently small  positive constant $\e_0,$ we have
$T^\ast=\infty,$ and \beq\label{eq1.14}
\begin{split}
 \|
u\|_{L^\infty(\R^+;H^1)}^2&+\bigl\|\f{u^r}r\bigr\|_{L^\infty(\R^+;L^2)}^2+\|\na u\|_{L^2(\R^+;H^1)}^2+\|\p_tu\|_{L^2(\R^+;L^2)}^2\\
&\qquad\qquad\qquad\qquad+\|\na\Pi\|_{L^2(\R^+;L^2)}^2\leq
C\cG_0 +1\with\\
\cG_0\eqdefa
&\exp\Bigl(C\|u_0\|_{L^2}^2\bigl(1+\|u_0\|_{L^2}^6\bigr)\Bigr)\bigl(\|
u_0\|_{H^1}^2+\bigl\|\f{u_0^r}{r}\bigr\|_{L^2}^2+2\|\Ga_0\|_{L^2}^2\bigr),\end{split}
\eeq and \beq \label{eq1.9}
\bigl\|\f{a}r\bigr\|_{L^\infty(\R^+;L^\infty)}\leq C
\bigl\|\f{a_0}r\bigr\|_{L^\infty}. \eeq

\item[(3)] Besides \eqref{eq1.8}, if $u_0\in L^p$ for some $p\in [1,2),$ let
$\beta(p)\eqdefa\frac{3}{4}\Bigl(\frac{2}{p}-1\Bigr),$ one has
\begin{equation}\label{eq4.1} \begin{split}
&\|u(t)\|_{L^2}^2\leq C\w{t}^{-2\beta(p)}, \quad \|\na u(t)\|_{L^2}^2\leq  C\w{t}^{-1-2\beta(p)},\\
& \|u_t(t)\|_{L^2}^2+ \|u(t)\|_{\dot{H}^2}^2+\|\na\Pi(t)\|_{L^2}^2
\leq C t^{-1}\w{t}^{-1-2\beta(p)}.
\end{split}
\eeq \end{itemize}
 }
\end{thm}

\begin{rmk}\label{rmk1.2}
\begin{itemize}

\item[(1)] Let us recall that the reason why one can prove the
global well-posdeness of classical 3-D Navier-Stokes system with
axisymmetric data and without swirl is that $\Ga\eqdefa\f{\om}r$
satisfies \beno
\p_t\Ga+u^r\p_r\Ga+u^z\p_z\Ga-\p_r^2\Ga-\p_z^2\Ga-\f3r\p_r\Ga=0,
\eeno which implies for all $p\in [1,\infty]$ that \beno
\|\Ga(t)\|_{L^p}\leq \|\Ga_0\|_{L^p}. \eeno Nevertheless in the case
of inhomogeneous Navier-Stokes system, $\Ga$ verifies \eqref{eq1.7}.
Then to get a global in time estimate for $\|\Ga(t)\|_{L^2},$ we
need the smallness condition \eqref{eq1.8}. We remark that in order
to prove the global regularity for the axisymmetric
Navier-Stokes-Boussinesq system without swirl, the authors
\cite{AHK2} require the support of the initial density $\r_0$ does
not intersect the axis $(Oz)$ and the projection of supp$\r_0$ on
the axis is a compact set, which seems stronger than \eqref{eq1.8}
near the axis $(Oz)$. Finally since we shall not use the vorticity
equation \eqref{eq1.6}, here we do not require the initial density
to be close enough to some positive constant.

\item[(2)] We remark that the decay estimates \eqref{eq4.1}  is in fact proved for
general global smooth solutions of \eqref{eq1.1}, which does not use
the axisymmetric structure of the solutions,  whenever $u_0\in L^p$
for some $p\in [1,2).$ In particular, we get rid of the technical
assumption in \cite{AGZ1} that \eqref{eq4.1} holds for $p\in
(1,6/5)$ and moreover the proof here is more concise than that in
\cite{AGZ1}. \end{itemize}
\end{rmk}

Let us complete this section with the notations we are going to use
in this context.

\smallbreak \noindent{\bf Notations:} $\dot{H}^s$ (resp. $H^s$)
denotes the homogeneous (resp. inhomogeneous) Sobolev space with
norm given by $\|f\|_{\dot{H}^s}\eqdefa
\Bigl(\int_{\R^3}|\xi|^{2s}|\widehat{f}(\xi)|^2\,d\xi\Bigr)^{\f12}$
(resp. $\|f\|_{{H}^s}\eqdefa
\Bigl(\int_{\R^3}\bigl(1+|\xi|^2\bigr)^{s}|\widehat{f}(\xi)|^2\,d\xi\Bigr)^{\f12}$).
For $X$ a Banach space and $I$ an interval of $\R,$ we denote by
${\mathcal{C}}(I;\,X)$ the set of continuous functions on $I$ with
values in $X.$  For $q\in[1,+\infty],$ the notation $L^q(I;\,X)$
stands for the set of measurable functions on $I$ with values in
$X,$ such that $t\longmapsto\|f(t)\|_{X}$ belongs to $L^q(I).$ Let
$R^2_+=(0,\infty)\times\R,$ we denote $\|f\|_{\wt{L}^q}\eqdefa
\bigl(\int_{\R^2_+}|f|^q\,dr\,dz\bigr)^{\f1q}.$ For $a\lesssim b$,
we mean that there is a uniform constant $C,$ which may be different
on different lines, such that $a\leq Cb$. We shall denote by $(a|b)$
(or $\int_{\R^3} a | b\, dx$) the $L^2(\R^3)$ inner product of $a$
and $b,$ and  finally $\wt{\na}\eqdefa (\p_r, \p_z).$

\renewcommand{\theequation}{\thesection.\arabic{equation}}
\setcounter{equation}{0}

\section{The  global $H^1$ estimate}

In this section, we shall prove the {\it a priori} globally in time
$H^1$ estimate for the velocity of \eqref{eq1.1} provided that there
holds \eqref{eq1.8}. Before proceeding, let us first rewrite the
momentum equation of \eqref{eq1.3}.

 Due to
$\pa_ru^r+\frac{u^r}{r}+\pa_zu^z=0$ and $\curl u=\omega e_{\theta}$
with $\omega\eqdefa \pa_z u^r-\pa_r u^z,$ we have
\begin{equation*}
\begin{split}
\frac{1}{r}\pa_r(r\pa_r u^r)+\pa_z^2u^r
-\frac{u^r}{r^2}=&-\frac{1}{r}\pa_r(r\pa_z u^z+u^r)+\pa_z^2u^r
-\frac{u^r}{r^2}\\
=&-\frac{1}{r}\bigl(r\pa_z \pa_r u^z-\frac{u^r}{r}
\bigr)+\pa_z^2u^r -\frac{u^r}{r^2}\\
=&\pa_z(\pa_z u^r-\pa_r u^z)=\pa_z \omega.
\end{split}
\end{equation*}
Similarly, one has
\begin{equation*}
\begin{split}
\frac{1}{r}\pa_r(r\pa_r u^z)+\pa_z^2u^z =&\pa_r^2 u^z+\frac{\pa_r
u^z}{r}-\pa_z\bigl(\pa_r u^r+\frac{u^r}{r}\bigr)\\
=&-\pa_r(\pa_z u^r-\pa_r u^z)
-\frac{1}{r}(\pa_z u^r-\pa_r u^z)\\
=&-\pa_r\omega -\frac{1}{r}\omega.
\end{split}
\end{equation*}
So that we can reformulate the momentum equation of \eqref{eq1.3} as
\begin{equation}\label{eq1.5}
\begin{cases}
&\r\pa_t u^r + \r u^r \pa_r u^r+\r u^{z} \pa_z u^r+\pa_r \Pi
- \pa_z \omega=0,\\
 &\r\pa_t
u^{z} + \r u^r \pa_r u^{z}+\r u^{z} \pa_z u^{z} +\pa_z
 \Pi+ \pa_r\omega +\frac{1}{r}\omega=0.
\end{cases}
\end{equation}

\subsection{Local in time $H^1$ estimate} The purpose of this subsection is
to present the estimate of $\|u\|_{L^\infty_T(H^1)}$ with $T$ going
to $\infty$ when $\e_0$ in \eqref{eq1.8} tending to zero.

\no $\bullet$ \underline{ $L^2$ energy estimate}

We first deduce from the transport equation of \eqref{eq1.3} and
\eqref{eq1.11} that \beq \label{eq2.1a}  m \leq \r(t,r,z)\leq  M.
\eeq While  by first multiplying the $u^r$ equation of \eqref{eq1.3}
by $u^r$ and then integrating the resulting equation over $\R^2_+$
with respect to the measure $r\,dr\,dz,$ we write \beno
\begin{split}
\f12\f{d}{dt}\int_{\R^2_+}\r&(u^r)^2\,rdr\,dz-\int_{\R^2_+}\bigl(r\p_t\r+\p_r(\r
u^r r)+\p_z(\r u^x
r)\bigr)(u^r)^2\,dr\,dz\\&-\int_{\R^2_+}\Pi\p_r(u^r r)\,dr\,dz
+\int_{R^2_+}\bigl((\p_ru^r)^2+(\p_zu^r)^2+\f{(u^r)^2}{r^2}\bigr)r\,dr\,dz=0.
\end{split}
\eeno Whereas using the transport equation and
$\p_r(u^rr)+\p_z(u^zr)=0$ of \eqref{eq1.3}, we find \beno
\begin{split}
r\p_t\r+\p_r(\r u^r r)+\p_z(\r u^z
r)=&r\bigl(\p_t\r+u^r\p_r\r+u^z\p_z
u^r)+\r\bigl(\p_r(u^rr)+\p_z(u^zr)\bigr)=0,
\end{split} \eeno so that we obtain
\beno
\begin{split}
\f12\f{d}{dt}\int_{\R^2_+}\r(u^r)^2\,rdr\,dz+\int_{R^2_+}\Bigl((\p_ru^r)^2+&(\p_zu^r)^2+\f{(u^r)^2}{r^2}\Bigr)r\,dr\,dz\\
 &\qquad\qquad=\int_{\R^2_+}\Pi\p_r(u^r r)\,dr\,dz .
\end{split}
\eeno Along the same line, we have \beno
\begin{split}
\f12\f{d}{dt}\int_{\R^2_+}\r(u^z)^2\,rdr\,dz+\int_{R^2_+}\bigl((\p_ru^z)^2+&(\p_zu^z)^2\bigr)r\,dr\,dz\\
 &\qquad\qquad=\int_{\R^2_+}\Pi\p_z(u^z r)\,dr\,dz .
\end{split} \eeno Hence due to $\p_r(ru^r)+\p_z(ru^z)=0,$ we achieve
\beno
\f12\f{d}{dt}\int_{\R^2_+}\r\bigl((u^r)^2+(u^z)^2\bigr)r\,dr\,dz+\int_{R^2_+}\Bigl(|\wt{\na}u^r|^2+&|\wt{\na}u^z|^2+\f{(u^r)^2}{r^2}\Bigr)r\,dr\,dz=0.
\eeno Integrating the above inequality over $[0,t]$ and using
\eqref{eq2.1a} gives rise to \beq\label{eq2.1}
\begin{split}
\|u\|_{L^\infty_t(L^2)}^2+\|\wt{\na}u\|_{L^2_t(L^2)}^2+\bigl\|\f{u^r}r\bigr\|_{L^2_t(L^2)}^2\leq
C\|u_0\|_{L^2}^2. \end{split} \eeq

\no$\bullet$ \underline{ $\dot{H}^1$ energy estimate}

By taking $L^2(\R^2_+,r\,dr\,dz)$ inner product of the $u^r$
equation of \eqref{eq1.3} with $\p_tu^r$ and using integration by
parts, we have \beno
\begin{split}
\f12\f{d}{dt}\int_{\R^2_+}\Bigl((\p_ru^r)^2&+(\p_zu^r)^2+\f{(u^r)^2}{r^2}\Bigr)r\,dr\,dz+\int_{\R^2_+}\r(\p_tu^r)^2r\,dr\,dz\\
&=-\int_{\R^2_+}\r\bigl( u^r\p_ru^r+u^z\p_zu^r\bigr)
\p_tu^rr\,dr\,dz+\int_{\R^2_+}\Pi \p_r(\p_tu^rr)\,dr\,dz.
\end{split}
\eeno Similarly we have \beno
\begin{split}
\f12\f{d}{dt}\int_{\R^2_+}\bigl((\p_ru^z)^2&+(\p_zu^z)^2\bigr)r\,dr\,dz+\int_{\R^2_+}\r(\p_tu^z)^2r\,dr\,dz\\
&=-\int_{\R^2_+}\r\bigl( u^r\p_ru^z+u^z\p_zu^z\bigr)
\p_tu^zr\,dr\,dz+\int_{\R^2_+}\Pi \p_z(\p_tu^zr)\,dr\,dz,
\end{split}
\eeno which together  $\p_r(ru^r)+\p_z(ru^z)=0$ gives rise to \beno
\begin{split}
\f12&\f{d}{dt}\int_{\R^2_+}\Bigl(|\wt{\na}u^r|^2+|\wt{\na}u^z|^2+\f{(u^r)^2}{r^2}\Bigr)r\,dr\,dz
+\int_{\R^2_+}\r\bigl((\p_tu^r)^2+(\p_tu^z)^2\bigr)r\,dr\,dz\\
=&-\int_{\R^2_+}\r\bigl( u^r\p_ru^r+u^z\p_zu^r\bigr)
\p_tu^rr\,dr\,dz-\int_{\R^2_+}\r\bigl( u^r\p_ru^z+u^z\p_zu^z\bigr)
\p_tu^zr\,dr\,dz\\
\leq &
C\Bigl(\|\sqrt{\r}u^r\p_ru^r\|_{L^2}^2+\|\sqrt{\r}u^z\p_zu^r\|_{L^2}^2+\|\sqrt{\r}u^r\p_ru^z\|_{L^2}^2+\|\sqrt{\r}u^z\p_zu^z)\|_{L^2}^2\Bigr)\\
&\qquad\qquad\qquad\qquad\qquad\qquad\qquad+\f12\bigl(\|\sqrt{\r}\p_tu^r\|_{L^2}^2+\|\sqrt{\r}\p_tu^z\|_{L^2}^2\bigr),
\end{split}
\eeno which along with \eqref{eq2.1a} implies \beq\label{eq2.2}
\begin{split}
\f{d}{dt}&\int_{\R^2_+}\Bigl(|\wt{\na}u^r|^2+|\wt{\na}u^z|^2+\f{(u^r)^2}{r^2}\Bigr)r\,dr\,dz
+\|\p_tu^r\|_{L^2}^2+\|\p_tu^z\|_{L^2}^2\\
 &\qquad\qquad\leq
C\Bigl(\|u^r\p_ru^r\|_{L^2}^2+\|u^z\p_zu^r\|_{L^2}^2+\|u^r\p_ru^z\|_{L^2}^2+\|u^z\p_zu^z)\|_{L^2}^2\Bigr).
\end{split}
\eeq

\no$\bullet$ \underline{The second derivative estimate of the
velocity}

By taking $L^2(\R^2_+;r\,dr\,dz)$ inner product of the $u^r$
equation of \eqref{eq1.5} with $\p_z\om$ and using integration by
parts, one has \beno
\begin{split}
\int_{\R^2_+}(\p_z\om)^2r\,dr\,dz=-\int_{\R^2_+}\p_z\p_r\Pi\ |\ \om
r\,dr\,dz-\int_{\R^2_+}\bigl(\r\p_tu^r+\r u^r\p_ru^r+\r
u^z\p_zu^r\bigr)\ |\ \p_z\om r\,dr\,dz. \end{split} \eeno Similarly
taking $L^2(\R^2_+;r\,dr\,dz)$ inner product of the $u^z$ equation
of \eqref{eq1.5} with $\p_r(r\om)r^{-1}$ leads to \beno
\begin{split}
\int_{\R^2_+}(\p_r(r\om))^2r^{-1}\,dr\,dz=&\int_{\R^2_+}\p_z\p_r\Pi\
|\ \om r\,dr\,dz\\
&-\int_{\R^2_+}\bigl(\r\p_tu^z+\r u^r\p_ru^z+\r u^z\p_zu^z\bigr)\ |
\p_r(\om r)\,dr\,dz. \end{split} \eeno Yet notice that \beno
\begin{split}
\int_{\R^2_+}(\p_r(r\om))^2r^{-1}\,dr\,dz=&\int_{\R^2_+}\bigl(\f{\om^2}r+2\om\p_r\om+(\p_r\om)^2r\bigr)\,dr\,dz\\
=&\int_{\R^2_+}\bigl(\f{\om^2}{r^2}+(\p_r\om)^2\bigr)r\,dr\,dz.
\end{split}
\eeno
 As a consequence, for $\Ga$ given by \eqref{eq1.7}, we obtain
\beq \label{eq2.3}
\begin{split}
\int_{\R^2_+}\bigl((\p_r\om)^2+(\p_z\om)^2+\Ga^2\bigr)r\,dr\,dz\leq
&
C\Bigl(\|u^r_t\|_{L^2}^2+\|u^z_t\|_{L^2}^2+\|u^r\p_ru^r\|_{L^2}^2\\
&+\|u^z\p_zu^r\|_{L^2}^2
+\|u^r\p_ru^z\|_{L^2}^2+\|u^z\p_zu^z)\|_{L^2}^2\Bigr).
\end{split}
\eeq Along the same line, we have \beq \label{eq2.4}
\begin{split}
\|\wt{\na}\Pi\|_{L^2}^2\leq
C\Bigl(&\|u^r_t\|_{L^2}^2+\|u^z_t\|_{L^2}^2+\|u^r\p_ru^r\|_{L^2}^2\\
&+\|u^z\p_zu^r\|_{L^2}^2+\|u^r\p_ru^z\|_{L^2}^2+\|u^z\p_zu^z\|_{L^2}^2\Bigr).
\end{split}
\eeq

\no$\bullet$ \underline{The combined estimate}

Let $\d>0$ be a small positive constant, which will be chosen
hereafter. By summing up \eqref{eq2.2} with $\d\times
\bigl(\eqref{eq2.3}+\eqref{eq2.4}\bigr)$ leads to \beno
\begin{split}
\f{d}{dt}&\int_{\R^2_+}\Bigl(|\wt{\na}u^r|^2+|\wt{\na}u^z|^2+\f{(u^r)^2}{r^2}\Bigr)r\,dr\,dz
+(1-2C\d)\bigl(\|\p_tu^r\|_{L^2}^2+\|\p_tu^z\|_{L^2}^2\bigr)\\
&+\d\Bigl(\int_{\R^2_+}\bigl(|\wt{\na}\om|^2+\Ga^2\bigr)r\,dr\,dz+\|\wt{\na}\Pi\|_{L^2}^2\Bigr)\\
 \leq &
C\Bigl(\|u^r\p_ru^r\|_{L^2}^2+\|u^z\p_zu^r\|_{L^2}^2+\|u^r\p_ru^z\|_{L^2}^2+\|u^z\p_zu^z\|_{L^2}^2\Bigr).
\end{split}
\eeno Taking $\d=\f1{4C}$ in the above inequality yields
\beq\label{eq2.5}
\begin{split}
\f{d}{dt}&\int_{\R^2_+}\Bigl(|\wt{\na}u^r|^2+|\wt{\na}u^z|^2+\f{(u^r)^2}{r^2}\Bigr)r\,dr\,dz
+\|\p_tu^r\|_{L^2}^2+\|\p_tu^z\|_{L^2}^2\\
&+\int_{\R^2_+}\bigl(|\wt{\na}\om|^2+\Ga^2\bigr)r\,dr\,dz+\|\wt{\na}\Pi\|_{L^2}^2\\
 \leq &
C\Bigl(\|u^r\p_ru^r\|_{L^2}^2+\|u^z\p_zu^r\|_{L^2}^2+\|u^r\p_ru^z\|_{L^2}^2+\|u^z\p_zu^z\|_{L^2}^2\Bigr).
\end{split}
\eeq

 In order to cope with the right hand side terms in \eqref{eq2.5}, we  take  cut-off functions $\varphi\in
C_0^\infty[0,\infty)$  and $\psi \in C^\infty[0,\infty)$ with
\beq\label{eq2.5a} \varphi(r)=\left\{\begin{array}{l}
\displaystyle 1\quad r\in [0,1/2],\\
\displaystyle 0\quad r\in [1,\infty),
\end{array}\right. \andf \psi(r)=\left\{\begin{array}{l}
\displaystyle 1\quad r\in [1/2, \infty),\\
\displaystyle 0\quad r\in [0,1/4),
\end{array}\right.\eeq
and present the  lemma as follows:
\begin{lem}\label{lem2.1}
{\sl Let $f(r,z)$ be a smooth enough function which decays
sufficiently fast at infinity.  Then for $\varphi(r)$ given by
\eqref{eq2.5a}, one has \beq\label{eq2.7a}
\int_{\R^2_+}f^4\varphi(r)r^3\,dr\,dz\leq
C\|f\|_{L^2}^2\bigl(\|f\|_{L^2}+\|\p_rf\|_{L^2}\bigr)\|\p_zf\|_{L^2}.
\eeq}
\end{lem}

\begin{proof} It is easy to observe that
\beno \begin{split} r^2f^2\varphi(r)\leq
&\int_0^\infty|\p_r(r^2f^2\varphi(r))|\,dr\\
\leq & C\int_0^\infty|f|(|f|+|\p_rf|)r\,dr, \end{split} \eeno and
\beno rf^2\leq \int_{\R}|\p_zf^2|r\,dz=2\int_{\R}|f||\p_zf|r\,dz,
\eeno  from which, we infer \beno
\begin{split}
\int_{\R^2_+}f^4\varphi(r)r^3\,dr\,dz\leq&
C\int_{\R^2_+}\int_0^\infty|f|(|f|+|\p_rf|)r\,dr\int_{\R}|f||\p_zf|r\,dz\,dr\,dz\\
\leq &C
\int_{\R^2_+}|f|(|f|+|\p_rf|)r\,dr\,dz\int_{\R^2_+}|f||\p_zf|r\,dr\,dz.
\end{split}
\eeno Applying H\"older inequality gives rise to \eqref{eq2.7a}.
\end{proof}

Now let us turn to the estimate of the nonlinear terms in
\eqref{eq2.5}. We first get, by applying H\"older's inequality and
the 2-D interpolation inequality, \beq\label{eq2.7b}
\|f\|_{L^4(\R^2)}\lesssim \|f\|_{L^2(\R^2)}^{\f12}\|\na
f\|_{L^2(\R^2)}^{\f12},\eeq that \beno
\begin{split}
\|u^r\p_ru\|_{L^2}^2\leq
&\|\p_ru\|_{\wt{L}^4}^2\|\sqrt{r}u^r\|_{\wt{L}^4}^2\\
\leq &C\Bigl(\int_{\R^3}\om^4
r^{-1}\,dx\Bigr)^{\f12}\|\sqrt{r}u^r\|_{\wt{L}^2}\|\wt{\na}(\sqrt{r}u^r)\|_{\wt{L}^2},
\end{split}
\eeno where we used Biot-Sarvart's law \beno
u(t,x)=\f1{4\pi}\int_{\R^3}\f{(y-x)\wedge e_\th
\om(t,y)}{|y-x|^3}\,dy \eeno and the fact that $r^{-1}$ is in $A^p$
class (see \cite{ga04} for instance) so that \beno
\|\p_ru\|_{\wt{L}^4}=\Bigl(\int_{\R^3}|\p_r
u|^4r^{-1}\,dx\Bigr)^{\f14}\leq
C\Bigl(\int_{\R^3}\om^4r^{-1}\,dx\Bigr)^{\f14}. \eeno  Then by
virtue of \eqref{eq2.7a} and \eqref{eq2.7b}, we infer \beno
\begin{split} \Bigl(\int_{\R^3}\om^4r^{-1}\,dx\Bigr)^{\f14}\leq
&\Bigl(\int_{\R^2_+}\Ga^4r^4\varphi(r)\,dr\,dz\Bigr)^{\f14}+
\Bigl(\int_{\R^2_+}\om^4(1-\varphi(r))\,dr\,dz\Bigr)^{\f14}\\
\lesssim &\Bigl(\int_{\R^2_+}\Ga^4r^3\varphi(r)\,dr\,dz\Bigr)^{\f14}+\|\om\psi\|_{\wt{L}^4}\\
\lesssim &\|\Ga \|_{L^2}^{\f12}\bigl(\|\Ga
\|_{L^2}^{\f12}+\|\wt{\na}\Ga \|_{L^2}^{\f12}\bigr)+\|\om \psi
\|_{\wt{L}^2}^{\f12}\|\wt{\na}(\om
\psi) \|_{\wt{L}^2}^{\f12}\\
\lesssim &\|\Ga \|_{L^2}^{\f12}\bigl(\|\Ga
\|_{L^2}^{\f12}+\|\wt{\na}\Ga \|_{L^2}^{\f12}\bigr)+
\|\om\|_{L^2}^{\f12}\bigl(\|\om\|_{L^2}^{\f12}+\|\wt{\na}\om\|_{L^2}^{\f12}\bigr).
\end{split}
\eeno Moreover, note that \beno
\begin{split}
&\|\wt{\na}(\sqrt{r}u^r)\|_{\wt{L}^2}\leq
C\bigl(\|\wt{\na}u^r\|_{L^2}+\bigl\|\f{u^r}r\bigr\|_{L^2}\bigr),
\end{split}
\eeno for any $\d>0,$  we write \beq\label{eq2.6}
\begin{split}
\|u^r\p_ru\|_{L^2}^2\leq &C\Bigl(\|\Ga \|_{L^2}\bigl(\|\Ga
\|_{L^2}+\|\wt{\na}\Ga \|_{L^2}\bigr)+
\|\om\|_{L^2}\bigl(\|\om\|_{L^2}+\|\wt{\na}\om\|_{L^2}\bigr)\Bigr)\\
&\times\|u^r\|_{L^2}\bigl(\|\wt{\na}u^r\|_{L^2}+\bigl\|\f{u^r}r\bigr\|_{L^2}\bigr)\\
\leq&
C_\d\|u^r\|_{L^2}^2\bigl(\|\wt{\na}u^r\|_{L^2}^2+\bigl\|\f{u^r}r\bigr\|_{L^2}^2\bigr)\bigl(\|\om\|_{L^2}^2+\|\Ga
\|_{L^2}^2\bigr)\\
&+\d \bigl(\|\om\|_{L^2}^2+\|\wt{\na}\om\|_{L^2}^2+\|\Ga
\|_{L^2}^2+\|\wt{\na}\Ga \|_{L^2}^2\bigr).
\end{split}
\eeq

To deal with $\|u^z\p_z u\|_{L^2},$ we split
$\int_{\R^2_+}(u^z\p_zu)^2r\,dr\,dz$ as \beq\label{eq2.7}
\int_{\R^2_+}(u^z\p_zu)^2r\,dr\,dz=\int_{\R^2_+}(u^z\p_zu)^2\varphi(r)r\,dr\,dz+
\int_{\R^2_+}(u^z\p_zu)^2(1-\varphi(r))r\,dr\,dz. \eeq By applying
\eqref{eq2.7b} and convexity inequality, we get for any $\d>0$
\beq\label{eq2.7ad}
\begin{split}
\int_{\R^2_+}&(u^z\p_zu)^2(1-\varphi(r))r\,dr\,dz\\
\lesssim
&\int_{\R^2_+}(u^z\p_zu\psi(r)r^{\f12})^2\,dr\,dz\\
\lesssim
&\bigl\|u^z\psi(r)r^{\f14}\bigr\|_{\wt{L}^4}^2\bigl\|\p_zu\psi(r)
r^{\f14}\bigr\|_{\wt{L}^4}^2\\
\lesssim
&\bigl\|u^z\psi(r)r^{\f14}\bigr\|_{\wt{L}^2}\bigl\|\wt{\na}(u^z\psi(r)r^{\f14})\bigr\|_{\wt{L}^2}\bigl\|\p_zu\psi
r^{\f14}\bigr\|_{\wt{L}^2}\bigl\|\wt{\na}(\p_zu\psi
r^{\f14})\bigr\|_{\wt{L}^2}\\
\leq
&C_\d\|u^z\|_{L^2}^2\bigl(\|u^z\|_{L^2}^2+\|\wt{\na}u^z\|_{L^2}^2\bigr)\|\p_zu\|_{L^2}^2+\d\bigl(\|\p_zu\|_{L^2}^2
+\|\wt{\na}\p_zu\|_{L^2}^2\bigr).
\end{split}
\eeq

Before proceeding, let us recall from (2.22) of \cite{MZ13} that
\beq\label{eq2.8}
\f{u^r}r=\p_z\D^{-1}\Ga-2\f{\p_r}r\D^{-1}\p_z\D^{-1}\Ga, \eeq and
from (21) of \cite{HR11} that \beq\label{eq2.8a}
\f{\p_r}r\D^{-1}W=\f{x_2^2}{r^2}\cR_{11}W+\f{x_1^2}{r^2}\cR_{22}W-2\f{x_1x_2}{r^2}\cR_{12}W
\eeq for every axisymmetric smooth function $W,$ and where
$\cR_{ij}\eqdefa \p_i\p_j\D^{-1}.$

By virtue of \eqref{eq2.7a}, we infer \beno
\begin{split}
\int_{\R^2_+}(u^z\p_zu^r)^2\varphi(r)r\,dr\,dz=&\int_{\R^2_+}(u^z)^2r^{\f32}\varphi^{\f12}(r)\bigl(\p_z\f{u^r}r\bigr)^2r^{\f32}\varphi^{\f12}(r)\,dr\,dz\\
\leq &
\Bigl(\int_{\R^2_+}(u^z)^4r^3\varphi(r)\,dr\,dz\Bigr)^{\f12}\Bigl(\int_{\R^2_+}\bigl(\p_z\f{u^r}r\bigr)^4r^3\varphi(r)\,dr\,dz\Bigr)^{\f12}\\
\lesssim
&\|u^z\|_{L^2}\bigl(\|u^z\|_{L^2}^{\f12}+\|\p_ru^z\|_{L^2}^{\f12}\bigr)\|\p_zu^z\|_{L^2}^{\f12}
\bigl\|\p_z\f{u^r}r\bigr\|_{L^2}\\
&\qquad\qquad\quad\times\bigl(\bigl\|\p_z\f{u^r}r\bigr\|_{L^2}^{\f12}+\bigl\|\p_z\p_r\f{u^r}r\bigr\|_{L^2}^{\f12}\bigr)
\|\p_z^2\f{u^r}r\bigr\|_{L^2}^{\f12}.
\end{split} \eeno
Yet it follows from \eqref{eq2.8} and \eqref{eq2.8a} that \beno
\begin{split}
\bigl\|\p_z\f{u^r}r\bigr\|_{L^2}\lesssim \|\Ga\|_{L^2},\quad
\|\p_z^2\f{u^r}r\bigr\|_{L^2}\lesssim \|\p_z\Ga\|_{L^2}\andf
\bigl\|\p_z\p_r\f{u^r}r\bigr\|_{L^2}\lesssim \|\wt{\na}\Ga\|_{L^2}.
\end{split}
\eeno Therefore, for any $\d>0,$ we have \beq\label{eq2.8bd}
\begin{split}
\int_{\R^2_+}(u^z\p_zu^r)^2\varphi(r)r\,dr\,dz\leq
C\|u^z\|_{L^2}^2\bigl(1+\|u^z\|_{L^2}^4\bigr)&\|\wt{\na}u^z\|_{L^2}^2\|\Ga\|_{L^2}^2\\
&+\d\bigl(
 \|\Ga\|_{L^2}^2+\|\wt{\na}\Ga\|_{L^2}^2\bigr).
 \end{split}
 \eeq

While since $\p_r(ru^r)+\p_z(ru^z)=0,$ we have \beq\label{eq2.8c}
\int_{\R^2_+}(u^z\p_zu^z)^2\varphi(r)r\,dr\,dz=\int_{\R^2_+}\bigl(u^z\bigl(\p_ru^r+\f{u^r}r\bigr)\bigr)^2\varphi(r)r\,dr\,dz.
\eeq Due to \eqref{eq2.8} and \eqref{eq2.8a}, we have \beno
\begin{split}
\int_{\R^2_+}\bigl(\f{u^z u^r}{r}\bigr)^2\varphi(r)r\,dr\,dz\leq
&\Bigl(\int_{\R^2_+}(u^z)^3\varphi^2(r)r\,dr\,dz\Bigr)^{\f23}\bigl\|\f{u^r}r\bigr\|_{L^6}^2\\
\leq
&\Bigl(\int_{\R^3}\f{(u^z)^3}{r^{\f32}}\,dx\Bigr)^{\f23}\bigl\|\p_z\D^{-1}\Ga\bigr\|_{L^6}^2\\
\leq  & C\|\na u^z\|_{L^2}^2\|\Ga\|_{L^2}^2.
\end{split}
\eeno where we used Sobolev-Hardy inequality from \cite{BT02} that
\beq\label{hardy} \int_{\R^N}\f{|u|^{q_\ast(s)}}{|x'|^s}\,dx\leq
C(s,q,N,k)\Bigl(\int_{\R^N}|\na u|^q\,dx\Bigr)^{\f{N-s}{N-q}}, \eeq
where $x=(x',z)\in R^N=R^k\times\R^{N-k}$ with $2\leq k\leq N,$
$1<q<N,$ $0\leq s\leq q$ and $s<k,$ $q_\ast\eqdefa\f{q(N-s)}{N-q},$
so that there holds \beno
\Bigl(\int_{\R^3}\f{(u^z)^3}{r^{\f32}}\,dx\Bigr)^{\f13}\leq C \|\na
u^z\|_{L^2}. \eeno

Whereas it follows from \eqref{eq2.8} that \beno
\p_ru^r=\p_z\D^{-1}\Ga+r\p_z\p_r\D^{-1}\Ga-2\p_r^2\D^{-1}\p_z\D^{-1}\Ga.
\eeno Applying Hardy's inequality \eqref{hardy} once again yields
\beno
\begin{split}
\int_{\R^2_+}(u^z)^2\bigl(\p_z\D^{-1}\Ga\bigr)^2\varphi(r)r\,dr\,dz\leq
&\Bigl(\int_{\R^2_+}|u^z|^3\varphi^{\f32}(r)r\,dr\,dz\Bigr)^{\f23}\|\p_z\D^{-1}\Ga\|_{L^6}^2\\
\leq
&\Bigl(\int_{\R^3}\f{|u^z|^3}{r^{\f32}}\,dx\Bigr)^{\f23}\|\Ga\|_{L^2}^2\\
\lesssim & \|\na u^z\|_{L^2}^2\|\Ga\|_{L^2}^2.
\end{split}
\eeno Similarly, by applying Lemma \ref{lem2.1}, one has
\beq\label{eq2.8b}
\begin{split}
\int_{\R^2_+}&(u^z)^2\bigl(r\p_z\p_r\D^{-1}\Ga\bigr)^2\varphi(r)r\,dr\,dz\\
\leq
&\Bigl(\int_{\R^2_+}(u^z)^4\varphi(r)r^3\,dr\,dz\Bigr)^{\f12}\Bigl(\int_{\R^2_+}|\p_z\p_r\D^{-1}\Ga|^4\varphi(r)r^3\,dr\,dz\Bigr)^{\f12}\\
\leq
&C\|u^z\|_{L^2}\bigl(\|u^z\|_{L^2}^{\f12}+\|\p_ru^z\|_{L^2}^{\f12}\bigr)\|\p_zu^z\|_{L^2}^{\f12}\|\Ga\|_{L^2}
\bigl(\|\Ga\|_{L^2}^{\f12}+\|\wt{\na}\Ga\|_{L^2}^{\f12}\bigr)\|\p_z\Ga\|_{L^2}^{\f12}\\
\leq & C_\d\|u^z\|_{L^2}^2\bigl(1+\|u^z\|_{L^2}^4\bigr)\|\wt{\na}
u^z\|_{L^2}^2\|\Ga\|_{L^2}^2+\d\bigl(\|\Ga\|_{L^2}^{2}+\|\wt{\na}\Ga\|_{L^2}^{2}\bigr).
\end{split}
\eeq Let $W\eqdefa \p_z\D^{-1}\Ga.$ Then by virtue of \eqref{eq2.8},
we find \beno
\begin{split}
\p_r^2\D^{-1}W=&\p_r\Bigl(\f{x_2^2}{r}\cR_{11}W+\f{x_1^2}{r}\cR_{22}W-2\f{x_1x_2}{r}\cR_{12}W\Bigr)\\
=&\sin^2\th\cR_{11}W+\cos^2\th\cR_{22}W-2\sin\th\cos\th
\cR_{12}W\\
&+r\bigl(\sin^2\th\p_r\cR_{11}W+\cos^2\th\p_r\cR_{22}W-2\sin\th\cos\th\p_r\cR_{12}W\bigr).
\end{split}
\eeno  It is easy to observe that \beno
\begin{split}
\int_{\R^2_+}&(u^z)^2\bigl(\sin^2\th\cR_{11}W+\cos^2\th\cR_{22}W-2\sin\th\cos\th\cR_{12}W\bigr)^2\varphi(r)r\,dr\,dz\\
\lesssim
&\Bigl(\int_{\R^3}|u^z|^3r^{-\f32}\,dx\Bigr)^{\f23}\|\p_z\D^{-1}\Ga\|_{L^6}^2
\lesssim  \|\na u^z\|_{L^2}^2\|\Ga\|_{L^2}^2,
\end{split}
\eeno and it follows from a similar derivation of \eqref{eq2.8b}
that \beno
\begin{split}
\int_{\R^2_+}&(u^z)^2\bigl(\sin^2\th\p_r\cR_{11}W+\cos^2\th\p_r\cR_{22}W-2\sin\th\cos\th\p_r\cR_{12}W)^2\varphi(r)r^3\,dr\,dz\\
\leq &
\Bigl(\int_{\R^2_+}|u^z|^4\varphi(r)r^{3}\,dr\,dz\Bigr)^{\f12}\\
&\quad\times \Bigl(\int_{R^2_+}
\bigl(\sin^2\th\p_r\cR_{11}W+\cos^2\th\p_r\cR_{22}W-2\sin\th\cos\th\p_r\cR_{12}W\bigr)^4\varphi(r)r^{3}\,dr\,dz\Bigr)^{\f12}\\
\leq & C_\d\|u^z\|_{L^2}^2\bigl(1+\|u^z\|_{L^2}^4\bigr)\|\wt{\na}
u^z\|_{L^2}^2\|\Ga\|_{L^2}^2+\d\bigl(\|\Ga\|_{L^2}^{2}+\|\wt{\na}\Ga\|_{L^2}^{2}\bigr)..
\end{split}
\eeno By resuming the above estimates into \eqref{eq2.8c}, we obtain
\beq\label{eq2.8d}
\int_{\R^2_+}(u^z\p_zu^z)^2\varphi(r)r\,dr\,dz\leq
C_\d\bigl(1+\|u^z\|_{L^2}^6\bigr)\|\wt{\na}u^z\|_{L^2}^2\|\Ga\|_{L^2}^2+
\d\bigl(\|\Ga\|_{L^2}^2+\|\wt{\na}\Ga\|_{L^2}^2\bigr). \eeq
 Therefore, by substituting the Estimates \eqref{eq2.7ad}, \eqref{eq2.8bd} and \eqref{eq2.8d} into \eqref{eq2.7},
  we obtain \beq\label{eq2.9} \begin{split}
 \|u^z\p_z
u\|_{L^2}^2\leq &
C_\d\Bigl(\bigl(1+\|u^z\|_{L^2}^6\bigr)\|\wt{\na}u^z\|_{L^2}^2\bigl(\|\wt{\na}u\|_{L^2}^2+\|\Ga\|_{L^2}^2\bigr)\\
&+(1+\|u^z\|_{L^2}^4)\|\p_zu\|_{L^2}^2\Bigr)+\d\bigl(\|\Ga\|_{L^2}^2+\|\wt{\na}\p_zu\|_{L^2}^2+\|\wt{\na}\Ga\|_{L^2}^2\bigr).
\end{split} \eeq

Note that for the axisymmetric flow, we have for $1<q<\infty$ \beq
\label{eq2.11}
\begin{aligned}
&(i)\quad\ \Vert\omega\Vert_{L^q}\quad \approx\quad
\Vert\nabla u\Vert_{L^q} \andf\\
&(ii)\quad \Vert\nabla\omega\Vert_{L^q} +
\bigl\|\frac{\omega}{r}\bigr\|_{L^q}\quad \approx \quad \|
\nabla^2u\|_{L^q}.
\end{aligned}
\eeq Thanks to \eqref{eq2.11}, by resuming the Estimates
\eqref{eq2.6} and \eqref{eq2.9} into \eqref{eq2.5} and taking $\d$
to be sufficiently small, we obtain \beq\label{eq2.10}
\begin{split}
\f{d}{dt}\Bigl(\|&\wt{\na}u(t)\|_{L^2}^2+\bigl\|\f{u^r(t)}{r}\bigr\|_{L^2}\Bigr)
+\|\p_tu\|_{L^2}^2+\|u\|_{\dot{H}^2}^2+\|\Ga\|_{L^2}^2+\|\wt{\na}\Pi\|_{L^2}^2\\
 \leq &
C_\d\Bigl((1+\|u\|_{L^2}^6)\bigl(\|\wt{\na}u\|_{L^2}^2+\bigl\|\f{u^r}r\bigr\|_{L^2}^2\bigr)
\bigl(\|\wt{\na}u\|_{L^2}^2+\|\Ga\|_{L^2}^2\bigr)\\
&\qquad\qquad\qquad\qquad\qquad\qquad\qquad\qquad+(1+\|u^z\|_{L^2}^4)\|\wt{\na}u\|_{L^2}^2\Bigr)+\d\|\wt{\na}\Ga\|_{L^2}^2.
\end{split}
\eeq  By applying Gronwall's inequality to \eqref{eq2.10}, we write
\beno
\begin{split}
\|&\na
u\|_{L^\infty_t(L^2)}^2+\bigl\|\f{u^r}r\bigr\|_{L^\infty_t(L^2)}^2+\|\p_tu\|_{L^2_t(L^2)}^2+\|
u\|_{L^2_t(\dot{H}^2)}^2+\|\Ga\|_{L^2_t(L^2)}^2+\|\na\Pi\|_{L^2_t(L^2)}^2\\
&\leq C\exp\Bigl(C\bigl(1+\|u\|_{L^\infty_t(L^2)}^6\bigr)\bigl(\|\na
u\|_{L^2_t(L^2)}^2+\bigl\|\f{u^r}r\bigr\|_{L^2_t(L^2)}^2\bigr)\Bigr)\\
&\  \times \Bigl(\|\na
u_0\|_{L^2}^2+\bigl\|\f{u_0^r}{r}\bigr\|_{L^2}^2+\bigl(1+\|u^z\|_{L^\infty_t(L^2)}^4\bigr)\|\wt{\na}
u\|_{L^2_t(L^2)}^2+\|\Ga\|_{L^\infty_t(L^2)}^2+\|\wt{\na}\Ga\|_{L^2_t(L^2)}^2\Bigr),
\end{split}
\eeno from which and \eqref{eq2.1}, we infer \beq\label{eq2.12}
\begin{split}
\|\na &
u\|_{L^\infty_t(L^2)}^2+\bigl\|\f{u^r}r\bigr\|_{L^\infty_t(L^2)}^2+\|\p_tu\|_{L^2_t(L^2)}^2+\|
u\|_{L^2_t(\dot{H}^2)}^2+\|\Ga\|_{L^2_t(L^2)}^2+\|\na\Pi\|_{L^2_t(L^2)}^2\\
\leq &C
\exp\Bigl(C\|u_0\|_{L^2}^2\bigl(1+\|u_0\|_{L^2}^6\bigr)\Bigr)\bigl(\|
u_0\|_{H^1}^2+\bigl\|\f{u_0^r}{r}\bigr\|_{L^2}^2+\|\Ga\|_{L^\infty_t(L^2)}^2+\|\na\Ga\|_{L^2_t(L^2)}^2\bigr).
\end{split}
\eeq

\no $\bullet$ \underline{The estimate of $\Ga$}

Let $a\eqdefa 1/\rho-1.$ Then we get, by taking $L^2$ inner product
of \eqref{eq1.7} with $\Ga$ and using integrating by parts, that
\beno
\begin{split}
\f12\f{d}{dt}&\|\Ga(t)\|_{L^2}^2+\int_{\R^2_+}\f1\r|\wt{\na}\Ga|^2r\,dr\,dz-2\int_{\R^2_+}\p_r\Bigl(\f{\Ga}\r\Bigr)\Ga\,dr\,dz\\
=&\int_{\R^2_+}a\bigl(\p_r\Pi\p_z\Ga-\p_z\Pi\p_r\Ga\bigr)\,dr\,dz\\
\leq
&\bigl\|\f{a}r\bigr\|_{L^\infty}\|\wt{\na}\Pi\|_{L^2}\|\wt{\na}\Ga\|_{L^2}.
\end{split}
\eeno Note that $a(t,0,z)=0,$ by using integration by parts, one has
\beno
\begin{split}
-2\int_{\R^2_+}\p_r\Bigl(\f{\Ga}\r\Bigr)\Ga\,dr\,dz=&-2\int_{\R^2_+}\p_r\Ga
\Ga\,dr\,dz-2\int_{\R^2_+}\p_r(a\Ga) \Ga\,dr\,dz\\
=&\int_{\R}\Ga^2(t,0,z)\,dz+2\int_{\R^2_+}a\Ga\p_r\Ga\,dr\,dz\\
\geq
&-C\bigl\|\f{a}r\bigr\|_{L^\infty}^2\|\Ga\|_{L^2}^2-\f14\bigl\|\f{\p_r\Ga}{\sqrt{\r}}\bigr\|_{L^2}^2.
\end{split}
\eeno Therefore due to \eqref{eq2.1a},  we infer \beq\label{eq2.14}
\begin{split}
\f{d}{dt}&\|\Ga(t)\|_{L^2}^2+\f1m\|\wt{\na}\Ga\|_{L^2}^2\leq
C\bigl\|\f{a}r\bigr\|_{L^\infty}^2\bigl(\|\wt{\na}\Pi\|_{L^2}^2+\|\Ga\|_{L^2}^2\bigr).
\end{split}
\eeq

On the other hand, it follows from the transport equation of
\eqref{eq1.3} that \beno &&\p_ta+u^r\p_ra+u^z\p_za=0\andf\\
&&\p_t\f{a}r+u^r\p_r\f{a}r+u^z\p_z\f{a}r+\f{u^r}r\f{a}r=0, \eeno
which yields \beq\label{eq2.14a}
\bigl\|\frac{a}{r}(t)\bigr\|_{L^\infty} \le
\bigl\|\frac{a_0}{r}\bigr\|_{L^\infty}\exp\Bigl(\bigl\|\frac{u^r}{r}\bigr\|_{L^1_t(L^\infty)}\Bigr).
\eeq
While note from \cite{AHK, Danchin} that
$$
\bigl\|\frac{u^r}{r}\bigr\|_{L^1_t(L^\infty)} \lesssim
\|\Ga\|_{L^1_t(L^{3,1})} \lesssim t^{\f34}
\|\Ga\|_{L^\infty_t(L^2)}^{\f12}\|\nabla\Ga\|_{L^2_t(L^2)}^{\f12}.
$$
So that by integrating \eqref{eq2.14} over $[0,t],$  we obtain \beno
\begin{split}
\|\Ga&\|_{L^\infty_t(L^2)}^2+\|\na\Ga\|_{L^2_t(L^2)}^2\leq
\|\Ga_0\|_{L^2}^2\\
&+C\bigl\|\f{a_0}r\bigr\|_{L^\infty}^2\exp\Bigl(Ct^{\f34}\|\Ga\|_{L^\infty_t(L^2)}^{\f12}\|\nabla\Ga\|_{L^2_t(L^2)}^{\f12}\Bigr)
\bigl(\|\na\Pi\|_{L^2_t(L^2)}^2+\|\Ga\|_{L^2_t(L^2)}^2\bigr).
\end{split}
\eeno Resuming the Estimate \eqref{eq2.12} into the above inequality
leads to \beq\label{eq2.15}
\begin{split}
\|\Ga&\|_{L^\infty_t(L^2)}^2+\|\na\Ga\|_{L^2_t(L^2)}^2\leq
\|\Ga_0\|_{L^2}^2+C\bigl\|\f{a_0}r\bigr\|_{L^\infty}^2\exp\Bigl(Ct^{\f34}\|\Ga\|_{L^\infty_t(L^2)}^{\f12}\|\nabla\Ga\|_{L^2_t(L^2)}^{\f12}\Bigr)\\
&\quad\times
\exp\Bigl(C\|u_0\|_{L^2}^2\bigl(1+\|u_0\|_{L^2}^6\bigr)\Bigr)\bigl(\|
u_0\|_{H^1}^2+\bigl\|\f{u_0^r}{r}\bigr\|_{L^2}^2+\|\Ga\|_{L^\infty_t(L^2)}^2+\|\na\Ga\|_{L^2_t(L^2)}^2\bigr).
\end{split}
\eeq

\begin{prop}\label{prop2.1}
{\sl Let $(\r, u, \na\Pi)$ be a smooth enough solution of
\eqref{eq1.3} on $[0,T^\ast),$ which satisfies \eqref{eq2.1a}. Let
$\cG_0$ be given by \eqref{eq1.14} and
\beq\label{eq2.16}\begin{split} t_1\eqdefa &
\biggl(\f1{2C\|\Ga_0\|_{L^2}}\ln\Bigl(\f{\|\Ga_0\|_{L^2}^2}{2C\bigl\|\f{a_0}r\bigr\|_{L^\infty}^2\cG_0}\Bigr)\biggr)^{\f43}.
\end{split}
\eeq Then under the assumption of \eqref{eq1.8},  one has
$T^\ast\geq t_1$ and there holds \ben &&\label{eq2.17}
\|\Ga\|_{L^\infty_{t_1}(L^2)}^2+\|\na\Ga\|_{L^2_{t_1}(L^2)}^2\leq 2
\|\Ga_0\|_{L^2}^2, \\
&& \label{eq2.18} \|\na
u\|_{L^\infty_{t_1}(L^2)}^2+\bigl\|\f{u^r}r\bigr\|_{L^\infty_{t_1}(L^2)}^2+\|\p_tu\|_{L^2_{t_1}(L^2)}^2+\|
u\|_{L^2_{t_1}(\dot{H}^2)}^2+\|\na\Pi\|_{L^2_t(L^2)}^2\leq C\cG_0.
\een
  }\end{prop}

\begin{proof}
Indeed if  $\bigl\|\f{a_0}r\bigr\|_{L^\infty}$ is sufficiently
small, we deduce from \eqref{eq2.15} and \eqref{eq2.16} that \beno
\|\Ga\|_{L^\infty_{t_1}(L^2)}^2+\|\na\Ga\|_{L^2_{t_1}(L^2)}^2\leq
\f32 \|\Ga_0\|_{L^2}^2. \eeno Substituting the above estimate into
 \eqref{eq2.12} gives rise to
\eqref{eq2.18}.  \eqref{eq2.18} together with the blow-up criteria
in \cite{kim06} implies that  $T^\ast\geq t_1.$
\end{proof}

\subsection{The global in time $H^1$ estimate} The goal of this
subsection is to present the global in time $H^1$ estimate for the
velocity field. Toward this, we first prove such a estimate for
small solutions of \eqref{eq1.1}, which does not use the
axisymmetric structure of the solutions.

\begin{lem}
\label{lem2.1} {\sl Let $(\r, u, \na\Pi)$ be a smooth enough
solution of \eqref{eq1.1} on $[0,T^\ast),$ which satisfies
\eqref{eq2.1a}.  Then there exist positive constants $\eta_1$  and
$\eta_2,$ which depend only on $\|u_0\|_{L^2},$ so that there holds
\begin{equation}\label{eq2.20}
\|\na u(t)\|^2_{L^2}+\int_{t_0}^{t}\bigl(m\|\partial_tu(t')\|_{
L^2}^2 +\eta_2\bigl(\|\na^2 u(t')\|_{
L^2}^2+\|\na\Pi(t')\|_{L^2}^2\bigr)\bigr)\,dt' \leq \|\na
u(t_0)\|_{L^2}^2
\end{equation} provided that
$\|\na u(t_0)\|_{L^2}\leq \eta_1.$ }
\end{lem}

\begin{proof} We first get, by
taking the  $L^2$ inner product of the momentum equations of
\eqref{eq1.1} with $\pa_t u$ and using integration by parts, that
\begin{equation*}
\begin{split}
\|\sqrt{\rho} \pa_t u(t)\|_{L^2}^2 +\frac{1}{2}\frac{d}{dt}\|\grad
u(t)\|^2_{L^2}=&-\bigl(\rho u \cdot \grad u\ |\ \pa_t u \bigr)_{L^2}\\
\leq & \|\sqrt{\rho}\|_{L^{\infty}}\|u\|_{L^3}\|\grad
u\|_{L^6}\|\sqrt{\rho}\pa_t u\|_{L^2}\\
 \leq & C\|u\|_{L^{2}}\|\grad u\|_{L^2}\|\grad^2
u\|_{L^2}^2+\frac{1}{4}\|\sqrt{\rho}\pa_t u\|_{L^2}^2,
\end{split}
\end{equation*}
which gives
\begin{equation*}
\f12\frac{d}{dt}\|\grad u(t)\|_{L^2}^2+\f34\|\sqrt{\rho}\pa_t
u(t)\|_{L^2}^2 \leq C \|u\|_{L^{2}}\|\grad u\|_{L^2}\|\grad^2
u\|_{L^2}^2.
\end{equation*}
On the other hand, it follows from the classical estimates on linear
Stokes operator and \beq\label{eq2.20a} \left\{\begin{array}{l}
\displaystyle  - \Delta u+ \grad\Pi=\r\pa_t u - \rho u \cdot\na u, \\
\displaystyle \dv\, u = 0,
\end{array}\right.
\eeq that \beno
\begin{split}
\|\na^2 u\|_{L^2}^2+\|\na\Pi\|_{L^2}^2\leq &
C\bigl(\|\r\p_tu\|_{L^2}^2+\|\r u\cdot\na u\|_{L^2}^2\bigr)\\
\leq&C\bigl(\|\sqrt{\r}\p_tu\|_{L^2}^2+\|\r\|_{L^\infty}\|u\|_{L^3}^2\|\na
u\|_{L^6}^2\bigr)\\
\leq &C\bigl(\|\sqrt{\r}\p_tu\|_{L^2}^2+\|u\|_{L^2}\|\na
u\|_{L^2}\|\na^2 u\|_{L^2}^2\bigr),
\end{split}
\eeno so that we obtain for any $\eta_2>0$ \beq\label{eq2.21}
\begin{split}\f12\frac{d}{dt}\|\grad u(t)\|_{L^2}^2+&\bigl(\f{3m}4-C\eta_2\bigr)\|\pa_t
u\|_{L^2}^2 \\
+&\bigl(\eta_2-C\|u_0\|_{L^2}\|\na u\|_{L^2}\bigr)\bigl(\|\na^2
u\|_{L^2}^2+\|\na\Pi\|_{L^2}^2\bigr)\leq 0. \end{split} \eeq

We denote
\begin{equation}\label{eq2.22}
 \tau^*\eqdefa\sup\bigl\{\ t\in [t_0, T^\ast)\, \big|\, \|\na u(t)\|_{L^2}\leq 2\eta_1\ \bigr\}.
\end{equation}
We claim that $\tau^\ast=T^\ast$ provided that $\eta_1$ is
sufficiently small. Indeed if $\tau^\ast<T^\ast,$ taking
$\eta_2=\f{m}{4C}$ and $\eta_1\leq \f{\eta_2}{2C\|u_0\|_{L^2}},$ we
deduce from \eqref{eq2.21} that
\begin{equation*}\label{eq2.19}
 \frac{d}{dt}\|\grad u(t)\|_{L^2}^2+m\|\pa_t u\|_{L^2}^2+
\eta_2\bigl(\|\grad^2 u\|_{L^2}^2+\|\na\Pi\|_{L^2}^2\bigr) \leq
0\qquad \mbox{for all}\  \, t\in [t_0, \tau^\ast),
\end{equation*}
 which implies
$$
\|\grad u(t)\|_{L^2}^2+\int_{t_0}^{\tau^\ast}\bigl(m\|\pa_t
u(t')\|_{L^2}^2+
 \eta_2(\|\grad^2 u(t')\|_{L^2}^2+\|\na\Pi(t')\|_{L^2}^2)\bigr)\,dt'\leq \|\grad
 u(t_0)\|_{L^2}^2\leq \eta^2_1.
 $$
 This contradict with \eqref{eq2.22}, and thus $\tau^\ast=T^\ast.$
 This
 concludes the proof of the lemma.
\end{proof}

\begin{prop}\label{prop2.2}
{\sl Let $(\r, u, \na\Pi)$ be the local unique smooth solution of
\eqref{eq1.3} on $[0,T^\ast),$ which satisfies \eqref{eq2.1a}. Then
$T^\ast=\infty$ and there holds \eqref{eq1.14} provided that $\e_0$
in \eqref{eq1.8} is sufficiently small. }\end{prop}

\begin{proof}
It follows from the derivation of \eqref{eq2.1} that
\beq\label{eq2.29} \f12\|\sqrt{\rho}u(t)\|_{L^2}^2+\int_0^t\|\nabla
u(t')\|_{L^2}^2\,dt' =\f12 \|\sqrt{\rho_0}u_0\|_{L^2}^2, \eeq which
ensures that for any positive integer $N,$ there holds
$$
\sum_{k=0}^{N-1}\int_k^{k+1}\|\nabla u(t')\|_{L^2}^2\,d t' \leq \f12
\|\sqrt{\rho_0}u_0\|_{L^2}^2.
$$
Thus there exists $0\le k_0\le N-1$ and some $t_0\in (k_0,k_0+1)$
such that
$$
\int_{k_0}^{k_0+1}\|\nabla u\|_{L^2}^2\,d\tau \leq
\frac{1}{2N}\|\sqrt{\rho_0}u_0\|_{L^2}^2 \andf \|\na
u(t_0)\|_{L^2}^2\leq \frac{1}{2N}\|\sqrt{\rho_0}u_0\|_{L^2}^2.
$$ For $\eta_1$ given by Lemma \ref{lem2.1},
taking $N$ so large that \beno \|\na u(t_0)\|_{L^2}^2\leq
\frac{1}{2N}\|\sqrt{\rho_0}u_0\|_{L^2}^2\leq \eta_1^2.\eeno Then we
deduce from Lemma \ref{lem2.1} that there holds \eqref{eq2.20}.

On the other hand, in view of \eqref{eq2.16}, we can take
$\bigl\|\f{a_0}r\bigr\|_{L^\infty}$ to be so small that $t_1\geq
t_0.$ Thus by summing up \eqref{eq2.18} and \eqref{eq2.20}, we
obtain for any $t<T^\ast,$ \beq\label{eq2.26} \begin{split} \|\na
u&\|_{L^\infty_t(L^2)}^2+\|\p_tu\|_{L^2_t(L^2)}^2+\|\na^2
u\|_{L^2_t(L^2)}^2+\|\na\Pi\|_{L^2_t(L^2)}^2\\
\leq & \|\na
u\|_{L^\infty(0,t_0;L^2)}^2+\|\p_tu\|_{L^2(0,t_0;L^2)}^2+\|\na^2
u\|_{L^2(0,t_0;L^2)}^2+\|\na\Pi\|_{L^2(0,t_0;L^2)}^2\\
&+\|\na
u\|_{L^\infty(t_0,t;L^2)}^2+\|\p_tu\|_{L^2(t_0,t;L^2)}^2+\|\na^2
u\|_{L^2(t_0,t;L^2)}^2+\|\na\Pi\|_{L^2(t_0,t;L^2)}^2\\
\leq &C\cG_0 +\eta_1,
\end{split}
\eeq for $\cG_0$ given by \eqref{eq1.14} and $\eta_1$ being
determined by Lemma \ref{lem2.1}. Then thanks to \eqref{eq2.26} and
the blow-up criteria in \cite{kim06}, we conclude that
$T^\ast=\infty.$ Moreover, by summing up \eqref{eq2.1} and
\eqref{eq2.26}, we achieve \eqref{eq1.14}. This finishes the proof
of Proposition \ref{prop2.2}.
\end{proof}

\renewcommand{\theequation}{\thesection.\arabic{equation}}
\setcounter{equation}{0}
\section{Decay estimates of the global solutions of \eqref{eq1.1}}

The purpose of this section is to present the decay estimates
\eqref{eq4.1} for any global smooth solutions of \eqref{eq1.1},
which does not use the particular axisymmetric structure of the
solutions.

\begin{lem}\label{lem3.2}
{\sl Let $(\r, u, \na\Pi)$ be a smooth enough solution of
\eqref{eq1.1} on $[0,T^\ast),$ which satisfies \eqref{eq2.1a}. Then
for $t<T^\ast,$ one has \beq\label{eq3.10}
\begin{split} \f{d}{dt}\|\na
u(t)\|_{L^2}^2&+\|\sqrt{\r}u_t(t)\|_{L^2}^2+\|u(t)\|_{\dot{H}^2}^2+\|\na\Pi(t)\|_{L^2}^2\leq
C\|\na u(t)\|_{{H}^1}^2\|\na u(t)\|_{L^2}^2,
\end{split}
\eeq and \beq\label{eq3.11} \begin{split}
\f{d}{dt}\|\sqrt{\r}u_t(t)\|_{L^2}^2&+\|\na u_t(t)\|_{L^2}^2\\
& \leq C\bigl(\|\na u(t)\|_{{H}^1}^2+\|u(t)\|_{\dot{H}^1}^4\bigr)
\bigl(\|\sqrt{\r}u_t(t)\|_{L^2}^2+\|\na u(t)\|_{L^2}^4\bigr).
\end{split}
\eeq }
\end{lem}
\begin{proof} We first get, by a similar derivation of
\eqref{eq2.21}, that \beno
\begin{split}
\f{d}{dt}\|\na u(t)\|_{L^2}^2+\bigl(\|\sqrt{\r}u_t\|_{L^2}^2+\|
u\|_{\dot{H}^2}^2+\|\na\Pi\|_{L^2}^2\bigr)\leq &
C\|\sqrt{\r}u\cdot\na
u\|_{L^2}^2\\
\leq & CM\|u\|_{L^6}^2\|\na u\|_{L^3}^2\\
\leq & CM \|\na u\|_{\dot{H}^{\f12}}^2\|\na u\|_{L^2}^2,
\end{split}
\eeno which gives \eqref{eq3.10}.

On the other hand, by taking $\p_t$ to the momentum equation of
\eqref{eq1.1}, we write \beno \r\bigl(\p_t u_t+u\cdot\na
u_t\bigr)-\D u_t+\na\Pi_t=-\r_t u_t-(\r u)_t\cdot\na u. \eeno Taking
$L^2$ inner product of the above equation with $u_t$ and using the
transport equation of \eqref{eq1.1}, we obtain \beq\label{eq3.12}
\f12\f{d}{dt}\|\sqrt{\r}u_t(t)\|_{L^2}^2+\|\na
u_t\|_{L^2}^2=-\int_{\R^3}\r_t|u_t|^2\,dx-\int_{\R^3}\r_t u\cdot\na
u\ |\ u_t\,dx-\int_{\R^3}\r u_t\cdot\na u\ |\ u_t\,dx. \eeq By using
the transport equation of \eqref{eq1.1} and integration by parts,
one has \beno
\begin{split}
\bigl|\int_{\R^3}\r_t|u_t|^2\,dx\bigr|=&\bigl|\int_{\R^3}\dive(\r
u)|u_t|^2\,dx\bigr|\\
=&2\bigl|\int_{\R^3}\r
u\cdot\na u_t\ |\ u_t\,dx\bigr|\\
\leq &2\sqrt{M}\|u\|_{L^\infty}\|\sqrt{\r}u_t\|_{L^2}\|\na
u_t\|_{L^2},
\end{split}
\eeno which together with the 3-D interpolation inequality that
\beq\label{eq3.13} \|u\|_{L^\infty}\leq C
\|u\|_{\dot{H}^1}^{\f12}\|u\|_{\dot{H}^2}^{\f12}, \eeq implies \beno
\bigl|\int_{\R^3}\r_t|u_t|^2\,dx\bigr|\leq
CM\|u\|_{\dot{H}^1}\|u\|_{\dot{H}^2}\|\sqrt{\r}u_t\|_{L^2}^2+\f16\|\na
u_t\|_{L^2}^2. \eeno Along the same line, we have \beno
\begin{split}
\int_{\R^3}&\r_tu\cdot\na u\ |\ u_t\,dx=-\int_{\R^3}\dive(\r
u)u\cdot\na u\ |\ u_t\,dx\\
=&\sum_{i,j,k=1}^3\Bigl(\int_{\R^3}\r
u^i\p_iu^j\p_ju^ku_t^k\,dx+\int_{\R^3}\r
u^iu^j\p_i\p_ju^ku_t^k\,dx+\int_{\R^3}\r
u^iu^j\p_ju^k\p_iu_t^k\,dx\Bigr).
\end{split}
\eeno Applying H\"older's inequality gives \beno
\begin{split}
\sum_{i,j,k=1}^3\bigl|\int_{\R^3}\r
u^i\p_iu^j\p_ju^ku_t^k\,dx\bigr|\leq &\sqrt{M}\|u\|_{L^\infty}\|\na
u\|_{L^3}\|\na u\|_{L^6}\|\sqrt{\r}u_t\|_{L^2}\\
\leq & C\bigl(\|u\|_{L^\infty}^2\|u\|_{\dot{H}^2}^2+\|\na
u\|_{\dot{H}^{\f12}}^2\|\sqrt{\r}u_t\|_{L^2}^2\bigr),
\end{split}
\eeno and \beno
\begin{split}
\sum_{i,j,k=1}^3\bigl|\int_{\R^3}\r
u^iu^j\p_i\p_ju^ku_t^k\,dx\bigr|\leq
&\sqrt{M}\|u\|_{L^\infty}^2\|\na^2 u\|_{L^2}\|\sqrt{\r}u_t\|_{L^2}\\
\leq &C\|u\|_{L^\infty}^2\bigl(\|
u\|_{\dot{H}^2}^2+\|\sqrt{\r}u_t\|_{L^2}^2\bigr),
\end{split}
\eeno and  \beno
\begin{split}
\sum_{i,j,k=1}^3\bigl|\int_{\R^3}\r
u^iu^j\p_ju^k\p_iu_t^k\,dx\bigr|\leq & M\|u\|_{L^6}^2\|\na
u\|_{L^6}\|\na u_t\|_{L^2}\\
\leq & C\|\na u\|_{L^2}^4\|u\|_{\dot{H}^2}^2+\f16\|\na
u_t\|_{L^2}^2.
\end{split}
\eeno This yields \beno \begin{split} \bigl|\int_{\R^3}\r_tu\cdot\na
u\ |\ u_t\,dx\bigr|\leq &C\bigl(\|u\|_{L^\infty}^2+\|\na
u\|_{\dot{H}^{\f12}}^2+\|\na
u\|_{L^2}^4\bigr)\bigl(\|u\|_{\dot{H}^2}^2+\|\sqrt{\r}u_t\|_{L^2}^2\bigr)+\f16\|\na
u_t\|_{L^2}^2.
\end{split}
\eeno Finally it is easy to observe that \beno
\begin{split}
\bigl|\int_{\R^3}\r u_t\cdot\na u\ |\ u_t\,dx\bigr|\leq
&\sqrt{M}\|u_t\|_{L^6}\|\na u\|_{L^3}\|\sqrt{\r}u_t\|_{L^2}\\
\leq &C\|\na
u\|_{\dot{H}^{\f12}}^2\|\sqrt{\r}u_t\|_{L^2}^2+\f16\|\na
u_t\|_{L^2}^2.
\end{split}
\eeno Resuming the above estimates into \eqref{eq3.12} and using
\eqref{eq3.13} results in \beq\label{eq3.14}
\begin{split}
\f{d}{dt}\|\sqrt{\r}u_t(t)\|_{L^2}^2&+\|\na u_t(t)\|_{L^2}^2\\
\leq & C\bigl(\|\na u(t)\|_{{H}^1}^2+\|u(t)\|_{\dot{H}^1}^4\bigr)
\bigl(\|u(t)\|_{\dot{H}^2}^2+\|\sqrt{\r}u_t(t)\|_{L^2}^2\bigr).
\end{split}
\eeq Whereas it follows from the classical estimates on linear
Stokes operator and \eqref{eq2.20a} that \beno
\begin{split}
\|u\|_{\dot{H}^2}+\|\na\Pi\|_{L^2}\leq &C\bigl(\|\r u_t\|_{L^2}+\|\r
u\cdot\na u\|_{L^2}\bigr)\\
\leq &C\bigl(\sqrt{M}\|\sqrt{\r}u_t\|_{L^2}+M\|u\|_{L^6}\|\na
u\|_{L^3}\bigr)\\
\leq &C\bigl(\|\sqrt{\r}u_t\|_{L^2}+\|\na
u\|_{L^2}^2\bigr)+\f12\|u\|_{\dot{H}^2},
\end{split}
\eeno which yields \beq\label{eq3.15}
\|u\|_{\dot{H}^2}+\|\na\Pi\|_{L^2}\leq
C\bigl(\|\sqrt{\r}u_t\|_{L^2}+\|\na u\|_{L^2}^2\bigr). \eeq
Substituting \eqref{eq3.15} into \eqref{eq3.14} leads to
\eqref{eq3.11}. This finishes the proof of the Lemma.
\end{proof}

\begin{col}\label{col3.1}
{\sl Under the assumptions of Lemma \ref{lem3.2} and that
\beq\label{eq3.15a} \|u\|_{L^\infty(0,T^\ast;H^1)}^2+\|\na
u\|_{L^2(0,T^\ast;H^1)}^2\leq C_0, \eeq
 one has for any $t<T^\ast,$ \beq\label{eq3.16}
\begin{split}
\w{t}\|\na
u(t)\|_{L^2}^2+\int_0^t\w{t'}\bigl(\|u_t(t')\|_{L^2}^2+\|u(t')\|_{\dot{H}^2}^2&+\|\na\Pi(t')\|_{L^2}^2\bigr)\,dt'\\
\leq &C\exp\bigl(CC_0)\|u_0\|_{H^1}^2\eqdefa C_1,
\end{split}
\eeq and \beq\label{eq3.17}
\begin{split}
t\w{t}\bigl(\|u_t(t)\|_{L^2}^2+\|u(t)\|_{\dot{H}^2}^2&+\|\na\Pi(t)\|_{L^2}^2\bigr)+\int_0^tt'\w{t'}\|\na
u_t(t')\|_{L^2}^2\,dt'\\
&\qquad\qquad\qquad\qquad\leq
CC_1(1+C_1)\exp\bigl(CC_0(1+C_0)\bigr)\eqdefa C_2.
\end{split}
\eeq } \end{col}

\begin{proof} We first get, by multiplying \eqref{eq3.10} by
$\w{t},$ that \beno \begin{split} \f{d}{dt}\bigl(\w{t}\|\na
u(t)\|_{L^2}^2&\bigr)+\w{t}\bigl(\|\sqrt{\r}u_t(t)\|_{L^2}^2+\|u(t)\|_{\dot{H}^2}^2+\|\na\Pi(t)\|_{L^2}^2\bigr)\\
&\qquad\qquad\leq \|\na u(t)\|_{L^2}^2+C\|\na
u(t)\|_{{H}^1}^2\w{t}\|\na u(t)\|_{L^2}^2.
\end{split}
\eeno Applying Gronwall's inequality and using \eqref{eq2.29},
\eqref{eq3.15a}  gives rise to \eqref{eq3.16}.

While multiplying \eqref{eq3.11} by $t\w{t}$ results in \beno
\begin{split}
\f{d}{dt}\bigl(t\w{t}&\|\sqrt{\r}u_t(t)\|_{L^2}^2\bigr)+t\w{t}\|\na
u_t(t)\|_{L^2}^2\leq
2\w{t}\|\sqrt{\r}u_t(t)\|_{L^2}^2\\
&\quad+C\bigl(\|\na u(t)\|_{{H}^1}^2+\|u(t)\|_{\dot{H}^1}^4\bigr)
t\w{t}\bigl(\|\sqrt{\r}u_t(t)\|_{L^2}^2+\|\na u(t)\|_{L^2}^4\bigr).
\end{split}
\eeno Applying Gronwall's inequality leads to \beno
\begin{split}
t&\w{t}\|\sqrt{\r}u_t(t)\|_{L^2}^2+\int_0^tt'\w{t'}\|\na
u_t(t')\|_{L^2}^2\,dt'\\
&\leq C\exp\Bigl(C\bigl(\|\na
u\|_{L^2_t({H}^1)}^2+\|u\|_{L^\infty_t(\dot{H}^1)}^2\|u\|_{L^2_t(\dot{H}^1)}^2\bigr)\Bigr)
\Bigl(\int_0^t\w{t'}\|\sqrt{\r}u_t(t')\|_{L^2}^2\,dt'\\
&\qquad\qquad\quad\qquad+\bigl\|\w{t'}\|\na
u(t')\|_{L^2}^2\bigr\|_{L^\infty_t}^2\bigl(\|\na
u\|_{L^2_t({H}^1)}^2
+\|u\|_{L^\infty_t(\dot{H}^1)}^2\|u\|_{L^2_t(\dot{H}^1)}^2\bigr)\Bigr),
\end{split}
\eeno from which, (\ref{eq3.15}-\ref{eq3.16}), we conclude the proof
of \eqref{eq3.17}.
\end{proof}

\begin{prop}\label{prop3.1} {\sl  Let $p\in [1,2)$ and $\beta(p)\eqdefa\frac{3}{4}\Bigl(\frac{2}{p}-1\Bigr).$ Then under the
assumptions of Corollary \ref{col3.1}, if we assume further that
$a_0\eqdefa \f1{\r_0}-1\in L^2(\R^3)$ and $u_0\in L^p(\R^3),$
there holds
\begin{equation}\label{eq3.1} \begin{split}
& \|u(t)\|_{L^2}\leq \left\{\begin{array}{l}
\displaystyle C\w{t}^{-\beta(p)} \qquad\ \mbox{if}\quad 1<p< 2, \\
\displaystyle C\w{t}^{-\bigl(\f34\bigr)_-}\qquad \mbox{if}\quad p=1,
\end{array}\right.
\end{split}
 \end{equation} for any $t<T^\ast,$ where the constant $C$ depends on $\|a_0\|_{L^2},$ $C_0, C_1$ and
 $C_2$ given by Corollary \ref{col3.1}.
 }
\end{prop}

\begin{proof} Motivated by \cite{AGZ1}, in order to use Schonbek's strategy in \cite{Schonbek}, we
split the phase-space $\R^3$ into two time-dependent regions so that
\begin{equation*} \|\grad u(t)\|_{L^2}^2 =
\int_{S(t)}|\xi|^2|\hat{u}(t,\xi)|^2 \,
d\xi+\int_{S(t)^{c}}|\xi|^2|\hat{u}(t,\xi)|^2 d\xi,
\end{equation*}
where $S(t)\eqdefa \{\xi: \ |\xi|\leq \sqrt{\f{M}{2}} \; g(t)\}$ and
$g(t)$ satisfies $g(t)\sim \langle{ t}\rangle^{-\frac{1}{2}},$ which
will be chosen later on.  Then due to the energy law \eqref{eq2.29}
of \eqref{eq1.1}, one has
\begin{equation}\label{eq3.2}
\frac{d}{dt}\|\sqrt{\rho}u(t)\|^2_{L^2}+g^2(t)\|
\sqrt{\rho}\;u(t)\|^2_{L^2} \leq
Mg^2(t)\int_{S(t)}|\hat{u}(t,\xi)|^2 \, d\xi
\end{equation}
To deal with the low frequency part of $u$ on the right-hand side of
\eqref{eq3.2},  we rewrite the momentum equations of \eqref{eq1.1}
as
\begin{equation*}
\begin{split} u(t)=&e^{t \Delta} u_0+\int_{0}^{t}e^{(t-t')
\Delta}\mathbb{P}\bigl(-\grad \cdot (u\otimes u)+a(\D u-\grad\Pi)
\bigr)(t')\,d t'.
\end{split}\end{equation*}
where $a\eqdefa\f1\r-1$ and  $\mathbb{P}\eqdefa Id-\na\D^{-1}\dive$
denotes the Leray projection
 operator.
Taking Fourier transform with respect to $x$ variables leads to
\begin{equation*}
\begin{aligned}
|\hat{u}(t,\xi)| \lesssim & e^{-t|\xi|^2} |\widehat{u}_0(\xi)|
+\int_{0}^{t}e^{-(t-t') |\xi|^2}\bigl(|\xi||\cF_x(u\otimes
u)|+|\cF_x(a(\D u-\grad\Pi))|\bigr)(t')\, d t',
\end{aligned}
\end{equation*}
which implies that
\begin{equation}\label{eq3.3}
\begin{split}
\int_{S(t)}|\hat{u}(t,\xi)|^2 d\xi \lesssim &\int_{S(t)}e^{-2 t
|\xi|^2} |\widehat{u}_0(\xi)|^2 d\xi+g^5(t)\Bigl(\int_{0}^{t}
\|\cF_x(u\otimes
u)(t')\|_{L_{\xi}^{\infty}} \,dt' \Bigr)^2\\
\qquad\qquad\qquad\qquad &\qquad\qquad\qquad\qquad+g^3(t)
\Bigl(\int_{0}^{t}\|\cF_x(a(\D u-\grad\Pi) )(t')\|_{L^\infty_\xi}
\,dt' \Bigr)^2.
\end{split}\end{equation}
Thanks to  \eqref{eq3.17}, we have
\begin{equation}\label{eq3.4}
\begin{split}
\Bigl(\int_{0}^{t}\|\cF_x(a (\D
u-\grad\Pi))(t')\|_{L_{\xi}^{\infty}}\, dt'\Bigr)^2 \leq &
\|a\|_{L^\infty_t(L^2)}^2\Bigl(\int_{0}^{t}\|(\D
u-\grad\Pi)(t')\|_{L^2} \, d t'\Bigr)^2\\ \lesssim &\|a_0\|_{L^2}^2
\Bigl(\int_0^t(t')^{-\f12}\w{t'}^{-\f12}\,dt'\Bigr)^{2}\lesssim
\ln^2\w{t}.
\end{split}
 \end{equation}
While it is easy to observe that
 \begin{equation*}\label{2.6}
\Bigl(\int_{0}^{t}\|\cF_x(u\otimes u)(t')\|_{L_{\xi}^{\infty}}\,
dt'\Bigr)^2 \leq \Bigl(\int_{0}^{t}\|u(t')\|_{L^{2}}^2 \,dt'
\Bigr)^2 \lesssim t^2\|u_0\|_{L^2}^2.
 \end{equation*}
  Note that for
$u_0 \in L^{p}(\R^3),$  let
$\frac{1}{q}\eqdefa\frac{4}{3}\beta(p)=\frac{2}{p}-1$ and
$\frac{1}{p}+\frac{1}{p'}=1,$ one has
\begin{equation}\label{eq3.8}
\begin{split}
\int_{S(t)}e^{-2t |\xi|^2} |\widehat{u}_0(\xi)|^2 \, d\xi \lesssim
&\Bigl(\int_{S(t)}e^{-2qt |\xi|^2} \, d\xi\Bigr)^{\frac{1}{q}}
\|\widehat{u}_0\|_{L^{p'}}^{2}\\
\lesssim & \|u_0\|_{L^p}^2\langle{ t}\rangle^{-2\beta(p)},
\end{split}
\end{equation}
where we used the  Hausd\"{o}rff-Young inequality in the last line
so that $\|\widehat{u}_0\|_{L^{p'}}\leq C\|u_0\|_{L^p}.$ Then since
$g(t) \lesssim \langle{ t}\rangle^{-\frac{1}{2}},$ we deduce from
\eqref{eq3.3} that
\begin{equation}\label{eq3.5}
 \int_{S(t)}|\hat{u}(t,\xi)|^2 \, d\xi
 \lesssim
 \langle{ t}\rangle^{-2\beta(p)}+\langle{ t}\rangle^{-\frac{1}{2}}
 \lesssim \left\{\begin{array}{l}
\displaystyle \w{t}^{-\f12}\qquad\ \mbox{if}\quad 1\leq p<\f32, \\
\displaystyle \w{t}^{-2\beta(p)}\quad\mbox{if}\quad \f32\leq p<2,
\end{array}\right.
\end{equation}
In the case when $\f32\leq p<2,$ by
 substituting \eqref{eq3.5} into \eqref{eq3.2}, we obtain
\begin{equation*}\frac{d}{dt}\|\sqrt{\rho}u(t)\|^2_{L^2}+g^2(t)\|
\sqrt{\rho}\;u(t)\|^2_{L^2} \lesssim g^2(t)\langle{
t}\rangle^{-2\beta(p)}\lesssim \langle{ t}\rangle^{-1-2\beta(p)},
\end{equation*}
from which, we infer
\begin{equation*}
 \begin{split}
e^{\int_{0}^{t}g^2(t')\, dt'} \|\sqrt{\rho}u(t)\|^2_{L^2} \lesssim
\|\sqrt{\rho_0}u_0\|^2_{L^2} + \int_{0}^{t}
e^{\int_{0}^{t'}g^2(t')dt' }\langle{ t'}\rangle^{-1-2\beta(p)}\,dt'.
\end{split}
\end{equation*}
Taking $\alpha > 2\beta(p)$ and $g^{2}(t)=\alpha\w{t}^{-1}$ in the
above inequality leads to
\begin{equation*}
\|\sqrt{\rho}u(t)\|^2_{L^2}\langle{ t}\rangle^{\alpha} \lesssim
1+\int_{0}^{t}\langle{ t'}\rangle^{\alpha-1-2\beta(p)}\,dt' \lesssim
1+\langle{ t}\rangle^{\alpha-2\beta(p)},
\end{equation*}
which yields \eqref{eq3.1} for $p\in [3/2,2).$

In the case when $1\leq p<\f32,$ by
 substituting the Estimate \eqref{eq3.5} into \eqref{eq3.2}, one has
\begin{equation*}\label{2.8}\frac{d}{dt}\|\sqrt{\rho}u(t)\|^2_{L^2}+g^2(t)\|
\sqrt{\rho}\;u(t)\|^2_{L^2} \lesssim g^2(t)\langle{
t}\rangle^{-\frac{1}{2}}\lesssim \langle{
t}\rangle^{-\frac{3}{2}},\end{equation*}
 which implies
 \begin{equation*}\label{2.10}
 \begin{split}
e^{\int_{0}^{t}g^2(t')\, dt'} \|\sqrt{\rho}u(t)\|^2_{L^2} \lesssim
\|\sqrt{\rho_0}u_0\|^2_{L^2} + \int_{0}^{t}
e^{\int_{0}^{t'}g^2(t')dt' }\langle{ t'}\rangle^{-\frac{3}{2}}\,dt'.
\end{split}
\end{equation*}
Taking $\alpha > \frac{1}{2}$ and $g^{2}(t)\eqdefa\alpha\w{t}^{-1}$
 in the above inequality results in
\begin{equation*}
\|\sqrt{\rho}u(t)\|^2_{L^2}\langle{ t}\rangle^{\alpha} \lesssim
1+\int_{0}^{t}\langle{ t'}\rangle^{\alpha-\frac{3}{2}}\,dt' \lesssim
1+\langle{ t}\rangle^{\alpha-\frac{1}{2}},
\end{equation*}
which gives
\begin{equation}\label{eq3.6}
\|u(t)\|_{L^2}\lesssim \langle{ t}\rangle^{-\frac{1}{4}}.
\end{equation}
Then by virtue of \eqref{eq3.6}, we write
 \begin{equation}\label{eq3.7}
 \begin{split}
\Bigl(\int_{0}^{t}\|\cF_x(u\otimes u)(t')\|_{L_{\xi}^{\infty}}\,
dt'\Bigr)^2 \leq  \Bigl(\int_{0}^{t}\|u(t')\|_{L^{2}}^2 \,dt'
\Bigr)^2 \lesssim \Bigl(\int_{0}^{t}\langle{
t'}\rangle^{-\frac{1}{2}} \,dt' \Bigr)^2 \lesssim \langle{
t}\rangle.
 \end{split}
 \end{equation}
Resuming the Estimates \eqref{eq3.4}, \eqref{eq3.8} and
\eqref{eq3.7} into \eqref{eq3.3} results in
 \begin{equation}\label{eq3.9}
 \int_{\bar{S}(t)}|\hat{u}(t,\xi)|^2 d\xi
 \lesssim
 \langle{ t}\rangle^{-2\beta(p)}+\langle{ t}\rangle^{-\bigl(\frac{3}{2}\bigr)_-}
 \lesssim \left\{\begin{array}{l}
\displaystyle \w{t}^{-2\beta(p)}\qquad\ \mbox{if}\quad 1<p<\f32, \\
\displaystyle \w{t}^{-\bigl(\f32\bigr)_-}\qquad\ \mbox{if}\quad p=1.
\end{array}\right.
\end{equation}
With \eqref{eq3.9}, we can repeat the previous argument to prove
\eqref{eq3.1} for the remaining  case when $p\in [1,3/2).$ This
completes the proof of the proposition.
\end{proof}

\begin{prop}\label{prop3.2}
{\sl Under the assumptions of Proposition \ref{prop3.1}, there holds
\eqref{eq4.1} for any $t<T^\ast.$ }
\end{prop}

\begin{proof} With Proposition \ref{prop3.1}, we shall use a  similar argument for the classical Navier-Stokes system
to derive the decay estimates for the derivatives of the velocity
(see \cite{HM06} for instance). In fact, for any $s<t<T^\ast,$ we
deduce from the energy equality of \eqref{eq1.1} that
\beq\label{eq3.19} \f12\|\sqrt{\r}u(t)\|_{L^2}^2+\int_s^t\|\na
u(t')\|_{L^2}^2\,dt'=\f12\|\sqrt{\r}u(s)\|_{L^2}^2. \eeq While
multiplying \eqref{eq3.10} by $(t-s)$ leads to \beno
\begin{split} \f{d}{dt}\bigl((t-s)\|\na
u(t)\|_{L^2}^2\bigr)&+(t-s)\bigl(\|\sqrt{\r}u_t(t)\|_{L^2}^2+\|\na^2u(t)\|_{L^2}^2+\|\na\Pi(t)\|_{L^2}^2\bigr)\\
&\quad\qquad\leq \|\na u(t)\|_{L^2}^2+ C\|\na
u(t)\|_{{H}^1}^2(t-s)\|\na u(t)\|_{L^2}^2,
\end{split}
\eeno Applying Gronwall's inequality and using \eqref{eq3.19}
results in \beno
\begin{split}
(t-s)\|\na u(t)\|_{L^2}^2\leq &\exp\Bigl(C\|\na
u\|_{L^2_t({H}^1)}^2\Bigr)\int_s^t\|\na
u(t')\|_{L^2}^2\,dt'\\
\leq &\f{\exp\bigl(CC_0\bigr)}2\|\sqrt{\r}u(s)\|_{L^2}^2.
\end{split}
\eeno In particular, taking $s=\f{t}2$ gives
$$\|\na u(t)\|_{L^2}^2\leq
C\w{t}^{-1}\|u(t/2)\|_{L^2}^2,$$ from which and \eqref{eq3.1}, we
infer for any $t<T^\ast$  \beq\label{eq3.20} \|\na
u(t)\|_{L^2}^2\leq C \left\{\begin{array}{l}
\displaystyle \w{t}^{-1-2\beta(p)} \qquad\ \mbox{if}\quad 1<p< 2, \\
\displaystyle \w{t}^{-\bigl(\f52\bigr)_-}\qquad\quad \mbox{if}\quad
p=1,
\end{array}\right. \eeq

Similarly by applying  Gronwall's lemma to \eqref{eq3.10} over
$[s,t],$ we write \beq\label{eq3.21} \begin{split} \|\na
u(t)\|_{L^2}^2+&\int_s^t\bigl(\|\sqrt{\r}u_t(t')\|_{L^2}^2+\|\na^2u(t')\|_{L^2}^2+\|\na\Pi(t')\|_{L^2}^2\bigr)\,dt'\\
\leq & \exp\Bigl(C\|\na u\|_{L^2_t({H}^1)}^2\Bigr)\|\na
u(s)\|_{L^2}^2\\
\leq & \exp\bigl(CC_0\bigr)\|\na u(s)\|_{L^2}^2.
\end{split}
\eeq Whereas by multiplying \eqref{eq3.11} by $(t-s)$ and applying
Gronwall's lemma to resulting inequality, we get \beno
\begin{split}
(t-s)\|\sqrt{\r}u_t(t)\|_{L^2}^2\leq
&\bigl(\int_s^t\|\sqrt{\r}u_t(t')\|_{L^2}^2\,dt'+\|\na
u\|_{L^\infty(s,t;L^2)}^4\bigr)\\
&\quad\times\exp\Bigl(C\bigl(\|\na
u\|_{L^2_t({H}^1)}^2+\|u\|_{L^\infty_t(\dot{H}^1)}^2
\|u\|_{L^2_t(\dot{H}^1)}^2\bigr)\Bigr)\\
\leq &\exp\bigl(CC_0(1+C_0)\bigr)\bigl(\|\na u(s)\|_{L^2}^2+\|\na
u\|_{L^\infty(s,t;L^2)}^4\bigr).
\end{split}
\eeno Taking $s=\f{t}2$ in the above inequality and using
\eqref{eq3.20}, we obtain \beno \|u_t(t)\|_{L^2}^2\leq C
\left\{\begin{array}{l}
\displaystyle t^{-1}\w{t}^{-1-2\beta(p)} \qquad\ \mbox{if}\quad 1<p< 2, \\
\displaystyle t^{-1}\w{t}^{-\bigl(\f52\bigr)_-}\qquad\quad\
\mbox{if}\quad p=1.
\end{array}\right. \eeno
which together with \eqref{eq3.15} and \eqref{eq3.20} ensures that
\beq \label{eq3.18}
\begin{split}
& \|u_t(t)\|_{L^2}^2+ \|u(t)\|_{\dot{H}^2}^2+\|\na\Pi(t)\|_{L^2}^2
\leq C \left\{\begin{array}{l}
\displaystyle Ct^{-1}\w{t}^{-1-2\beta(p)} \qquad\ \mbox{if}\quad 1<p< 2, \\
\displaystyle t^{-1}\w{t}^{-\bigl(\f52\bigr)_-}\qquad\qquad
\mbox{if}\quad p=1,
\end{array}\right. \end{split} \eeq
for any $t<T^\ast.$

With \eqref{eq3.20} and \eqref{eq3.18}, it remains to prove
\eqref{eq4.1}  for $p=1.$ As a matter of fact, we first deduce from
\eqref{eq3.18} that
\begin{equation*}
\begin{split}
\Bigl(\int_{0}^{t}\|\cF_x(a (\D
u-\grad\Pi))(t')\|_{L_{\xi}^{\infty}}\, dt'\Bigr)^2 \leq &
\|a\|_{L^\infty_t(L^2)}^2\Bigl(\int_{0}^{t}\|(\D
u-\grad\Pi)(t')\|_{L^2} \, d t'\Bigr)^2\\ \lesssim & \|a_0\|_{L^2}^2
\Bigl(\int_0^t(t')^{-\f12}\w{t'}^{-\bigl(\f54\bigr)_-}\,dt'\Bigr)^{2}\leq
C.
\end{split}
 \end{equation*}
With \eqref{eq3.4} being replaced by the above inequality, by
repeating the proof of Proposition \ref{prop3.1}, we can prove the
first inequality of \eqref{eq4.1} for $p=1.$  Then repeating the
proof of  \eqref{eq3.18}, we conclude the proof of the remaining two
inequalities in \eqref{eq4.1} for $p=1.$ This finishes the proof of
Proposition \ref{prop3.2}.
\end{proof}

\renewcommand{\theequation}{\thesection.\arabic{equation}}
\setcounter{equation}{0}
\section{The proof of Theorem \ref{thm1.1}}

The goal of this section is to complete the proof of Theorem
\ref{thm1.1}. In order to do so, we first prove the following
globally in time Lipschitz estimate for the convection velocity
field, which will be used to prove the propagation of the size for
$\bigl\|\f{a_0}r\|_{L^\infty}.$

\begin{lem}\label{lem4.2}
{\sl  Let $(\r,u,\na\Pi)$ be a c smooth enough axisymmetric solution
of \eqref{eq1.1} on $[0,T^\ast).$ Then  under the assumptions
\eqref{eq1.11} and \eqref{eq1.8}, we have $T^\ast=\infty,$ and there
holds \beq \|\na u\|_{L^1(\R_+;L^\infty)}\leq C, \label{eq4.2} \eeq
for some positive constant depending on $m, M$ and $\|u_0\|_{H^1}.$}
\end{lem}
\begin{proof} Under the assumptions of \eqref{eq1.11} and
\eqref{eq1.8}, we deduce from Proposition \ref{prop2.2} that
$T^\ast=\infty$ and moreover Corollary \ref{col3.1} ensures that
 \beq\label{eq4.2a}
\begin{split}
&\sup_{t\in [0,\infty)}\w{t}\|\na
u(t)\|_{L^2}^2+\int_0^\infty\w{t}\bigl(\|u_t(t)\|_{L^2}^2+\|u(t)\|_{\dot{H}^2}^2+\|\na\Pi(t)\|_{L^2}^2\bigr)\,dt
\leq  C_1,\\
&\sup_{t\in [0,\infty)}
t\w{t}\bigl(\|u_t(t)\|_{L^2}^2+\|u(t)\|_{\dot{H}^2}^2+\|\na\Pi(t)\|_{L^2}^2\bigr)+\int_0^\infty
t\w{t}\|\na u_t(t)\|_{L^2}^2\,dt\leq C_2,
\end{split}
\eeq where $C_1$ and $C_2$ given by \eqref{eq3.16} and
\eqref{eq3.17} respectively. In particular, by using Sobolev
imbedding theorem, we obtain \beq\label{eq4.3} \int_0^\infty
t\w{t}\| u_t(t)\|_{L^6}^2\,dt\leq C_2. \eeq

On the other hand, in view of \eqref{eq2.20a}, we deduce from the
classical estimates of linear Stokes operator that \beno \|\na^2
u(t)\|_{L^6}+\|\na\Pi(t)\|_{L^6}\leq
C\bigl(\|u_t(t)\|_{L^6}+\|u(t)\|_{L^\infty}\|\na u(t)\|_{L^6}\bigr),
\eeno which  together with \eqref{eq3.13} yields \beno
\begin{split}
\int_0^\infty & t\w{t}\bigl(\|\na^2
u(t)\|_{L^6}^2+\|\na\Pi(t)\|_{L^6}^2\bigr)\,dt\\
&\leq C\Bigl(\int_0^\infty t\w{t}\|
u_t(t)\|_{L^6}^2\,dt+\int_0^\infty t\w{t}\| u(t)\|_{\dot{H}^1}\|
u(t)\|_{\dot{H}^2}^3\,dt\Bigr).
\end{split} \eeno
Yet it follows from \eqref{eq4.2a} that $$ \int_0^\infty t\w{t}\|
u(t)\|_{\dot{H}^1}\| u(t)\|_{\dot{H}^2}^3\,dt\leq
C\sqrt{C_1}C_2^{\f32}\int_0^\infty t^{-\f12}\w{t}^{-1}\,dt\leq
C\sqrt{C_1}C_2^{\f32},
$$
which together with \eqref{eq4.3} ensures that \beq \label{eq4.4}
\int_0^\infty t\w{t}\bigl(\|\na^2
u(t)\|_{L^6}^2+\|\na\Pi(t)\|_{L^6}^2\bigr)\,dt\leq
CC_2\bigl(1+\sqrt{C_1C_2}\bigr)\eqdefa C_3. \eeq  By virtue of
\eqref{eq4.2a} and \eqref{eq4.4}, we infer \beno
\begin{split}
\int_0^\infty\|\na u(t)\|_{L^\infty}\,dt\leq
&C\int_0^\infty\|u(t)\|_{\dot{H}^2}^{\f12}\|\na^2u(t)\|_{L^6}^{\f12}\,dt\\
\leq
&CC_2^{\f14}\int_0^\infty t^{-\f12}\w{t}^{-\f12}\bigl(t\w{t}\|\na^2 u(t)\|_{L^6}^2\bigr)^{\f14}\,dt\\
\leq &CC_2^{\f14}\Bigl(\int_0^\infty
t^{-\f23}\w{t}^{-\f23}\,dt\Bigr)^{\f34}\Bigl(\int_0^\infty
t\w{t}\|\na^2 u(t)\|_{L^6}^2\,dt\Bigr)^{\f14}\\
\leq &CC_2^{\f14}C_3^{\f14}.
\end{split}
\eeno This gives rise to \eqref{eq4.2}.
\end{proof}

Now we are in a position to complete the proof of Theorem
\ref{thm1.1}.

\begin{proof}[Proof of Theorem \ref{thm1.1}] The general strategy to
prove the existence result to a nonlinear partial differential
equation is first to construct an appropriate approximate solutions,
and then perform the uniform estimates to these approximate solution
sequence, and finally the existence result follows from a
compactness argument. For simplicity, here we just present the {\it
a priori} estimates to smooth enough solutions of \eqref{eq1.3}.

Given axisymmetric initial data $(\r_0,u_0)$ with $\r_0$ satisfying
\eqref{eq1.11} and $a_0\in L^2\cap L^\infty,$ $\f{a_0}r\in
L^\infty,$ $u_0\in H^1,$ we deduce from \eqref{eq2.10} and
\eqref{eq2.14} that there exists a maximal positive time $T^\ast$ so
that \eqref{eq1.3} has a solution on $[0,T^\ast)$ which satisfies
for any $T<T^\ast,$ \beno \|u\|_{L^\infty_T(H^1)}+\|\na
u\|_{L^2_T(H^1)}+\|\p_tu\|_{L^2_T(L^2)}+\|\na\Pi\|_{L^2_T(L^2)}+
\|\Ga\|_{L^\infty_T(L^2)}+\|\na\Ga\|_{L^2_T(L^2)}\leq C, \eeno from
which and Corollary \ref{col3.1}, we deduce that there holds
\eqref{eq1.13}. And hence the uniqueness part of Theorem
\ref{thm1.1} follows from the uniqueness result in \cite{PZZ1}.

Now if $T^\ast<\infty$ and  there holds \beno \lim_{t\to
T^\ast}\bigl\|\f{a(t)}r\bigr\|_{L^\infty}=C_\ast<\infty. \eeno Let
us take $\d$ so small that \beno 2mCC_\ast\leq \f12. \eeno Then we
get, by summing up \eqref{eq2.10} and $2m \d\times$ \eqref{eq2.14},
that \beno
\begin{split}
\f{d}{dt}&\Bigl(\|\wt{\na}
u(t)\|_{L^2}^2+\bigl\|\f{u^r(t)}{r}\bigr\|_{L^2}^2+2m\d\|\Ga(t)\|_{L^2}^2\Bigr)
\\
&\qquad+\|\p_tu\|_{L^2}^2+\|u\|_{\dot{H}^2}^2+\f12\bigl(\|\wt{\na}\Pi\|_{L^2}^2+\|\Ga\|_{L^2}^2\bigr)+\d\|\wt{\na}\Ga\|_{L^2}^2\\
 \leq &
C_\d\Bigl((1+\|u\|_{L^2}^6)\bigl(\|\wt{\na}u\|_{L^2}^2+\bigl\|\f{u^r}r\bigr\|_{L^2}^2\bigr)
\bigl(\|\wt{\na}u\|_{L^2}^2+\|\Ga\|_{L^2}^2\bigr)+(1+\|u^z\|_{L^2}^4)\|\p_zu\|_{L^2}^2\Bigr).
\end{split}
\eeno Applying Gronwall's inequality and using \eqref{eq2.1} leads
to $$\longformule{
\|\wt{\na}u\|_{L^\infty_T(L^2)}^2+\bigl\|\f{u^r}{r}\bigr\|_{L^\infty_T(L^2)}^2+\|\Ga\|_{L^\infty_T(L^2)}^2
+\|\p_tu\|_{L^2_T(L^2)}^2+\|\wt{\na}\Pi\|_{L^2_T(L^2)}^2}{{}+\|u\|_{L^2_T(\dot{H}^2)}^2+\|\Ga\|_{L^2_T(L^2)}^2+\|\wt{\na}\Ga\|_{L^2_T(L^2)}^2\leq
C, } $$ for any $T<T^\ast.$ Therefore we can extend the solution
beyond the time $T^\ast,$ which contradicts with the maximality of
$T^\ast.$ Hence there holds \eqref{eq4.6}.

Under the assumption of \eqref{eq1.8}, we deduce from Proposition
\ref{prop2.2} that $T^\ast=\infty$ and there holds \eqref{eq1.14}.
Moreover, Lemma \ref{lem4.2} ensures that \beno \|\na
u\|_{L^1(\R^+;L^\infty)}\leq C, \eeno which together with
\eqref{eq2.14a} and
$$\bigl\|\frac{u^r}{r}\bigr\|_{L^1(\R^+;L^\infty)}\leq \|\na u
\|_{L^1(\R^+;L^\infty)}$$ gives rise to \eqref{eq1.9}.

Finally with additional assumption that $u_0\in L^p$ for some $p\in
[1,2),$ we deduce from Proposition \ref{prop3.2} that there holds
the decay estimate \eqref{eq4.1}. This finishes the proof of Theorem
\ref{thm1.1}.
\end{proof}

\noindent {\bf Acknowledgments.} We would like to thank
Rapha$\ddot{e}$l Danchin and Guilong Gui for profitable discussions
on this topic.

Part of this work was done when we were visiting Morningside Center
of Mathematics, CAS, in the summer of 2014. We appreciate the
hospitality and the financial support from the Center.  P. Zhang is
partially supported by NSF of China under Grant 11371037, the
fellowship from Chinese Academy of Sciences and innovation grant
from National Center for Mathematics and Interdisciplinary Sciences.


\begin{thebibliography}{50}
\bibitem{Abidi}
H. Abidi, R\'esultats de r\'egularit\'e de solutions
axisym\'etriques pour le syst$\grave{e}$me de Navier-Stokes, {\it
Bull. Sci. Math.}, {\bf 132} (2008),  592-624.

\bibitem{AHK}
H. Abidi, T. Hmidi and S. Keraani, On the global well-posedness for
the axisymmetric Euler equations, {\it  Math. Ann.} {\bf 347}
(2010), 15--41.

\bibitem{AHK2}
H. Abidi, T. Hmidi and S. Keraani, On the global regularity of
axisymmetric Navier-Stokes-Boussinesq system, {\it Discrete Contin.
Dyn. Syst.}, {\bf 29} (2011), 737-756.

\bibitem{AGZ1} H. Abidi, G. Gui and P. Zhang, Stability to
the global large solutions of the 3-D inhomogeneous Navier-Stokes
equations,  {\it Comm. Pure. Appl. Math.},   {\bf 64} (2011),
832-881.

 \bibitem{AGZ2} H. Abidi, G. Gui and P. Zhang,
Well-posedness of 3-D inhomogeneous Navier-Stokes equations with
highly oscillatory initial velocity field, {\it J. Math. Pures
Appl.}, {\bf (9) 100} (2013), 166-203.

\bibitem{BT02} M. Badiale and G. Tarantello,  A Sobolev-Hardy inequality
with applications to a nonlinear elliptic equation arising in
astrophysics, {\it Arch. Ration. Mech. Anal.}, {\bf 163} (2002),
259-293.

\bibitem{CL} D. Chae and J.  Lee,  On the regularity of the axisymmetric
solutions of the Navier-Stokes equations, {\it Math. Z.}, {\bf  239}
(2002),  645-671.

\bibitem{Danchin} R.   Danchin,  Axisymmetric incompressible flows with bounded vorticity,
{\it Russian Math.   Surveys}, {\bf 62} (2007),  73--94.

\bibitem{Danchin04} R. Danchin,  Local and global well-posedness results for flows of
inhomogeneous viscous fluids., {\it Adv. Differential Equations},
{\bf 9} (2004),  353-386.

\bibitem{dm} R. Danchin and P.~B.~ Mucha, A Lagrangian approach for
the incompressible Navier-Stokes equations with variable density,
{\it Comm. Pure. Appl. Math.}, {\bf 65}  (2012), 1458--1480.

\bibitem{dz} R. Danchin and P. Zhang,  Inhomogeneous Navier¨CStokes equations in the half-space, with only bounded density,
{\it  J. Funct. Anal.}, {\bf 267} (2014),  2371-2436.


\bibitem{ga04} L. Grafakos, {\it Classical and modern Fourier analysis,}  Pearson
Education, Inc., Upper Saddle River, NJ, 2004.


\bibitem{HM06} C. He and T. Miyakawa,  On two-dimensional Navier-Stokes flows with rotational symmetries, {\it
Funkcial. Ekvac.}, {\bf 49}  (2006),  163-192.


\bibitem{HR11} T. Hmidi and F.  Rousset,  Global well-posedness for the
Euler-Boussinesq system with axisymmetric data, {\it J. Funct.
Anal.}, {\bf 260} (2011),  745-796.



\bibitem{kim06} H. Kim, A blow-up criterion for the nonhomogeneous
incompressible Navier-Stokes equations, {\it SIAM J. Math. Anal.},
{\bf  37} (2006),  1417-1434.

\bibitem{La} O.~ A.  Lady$\breve{z}$enskaja, Unique global solvability of the three-dimensional Cauchy problem for the Navier-Stokes
 equations in the presence of axial symmetry, (Russian) {\it Zap. Nau$\breve{c}$n. Sem. Leningrad. Otdel. Mat. Inst. Steklov. (LOMI)},
 {\bf 7} (1968), 155-177.

\bibitem{LS} O. Ladyzhenskaya and V. Solonnikov:  The unique solvability of an initial-boundary value problem for viscous incompressible inhomogeneous fluids, {\em
Journal of Soviet Mathematics}, {\bf  9} (1978),  697--749.

\bibitem{LMNP} S. Leonardi, J. M\'{a}lek, J.  Ne\u{c}as and M. Pokorny,  On axially
symmetric flows in $\mathbb{R}^3,$ {\it  Z. Anal. Anwendungen}, {\bf
18} (1999),  639-649.

\bibitem{LP}  P.~L.~ Lions:   {\it Mathematical topics in fluid mechanics.
Vol. 1. Incompressible models}, Oxford Lecture Series in Mathematics
and its Applications, 3. Oxford Science Publications. The Clarendon
Press, Oxford University Press, New York, 1996.


\bibitem{MZ13} C. Miao and X.  Zheng,  On the global well-posedness for the
Boussinesq system with horizontal dissipation, {\it Comm. Math.
Phys.}, {\bf  321} (2013),  33-67.

\bibitem{PZZ1} M. Paicu, P. Zhang and Z. Zhang, Global well-posedness of inhomogeneous Navier-Stokes equations
with bounded density, {\it Comm. Partial Differential Equations},
{\bf 38} (2013), 1208-1234.

\bibitem{Schonbek}  M. Schonbek, Large time behavior of solutions to Navier-Stokes
equations, {\it Comm.in P. D. E.}, {\bf 11}  (1986), 733--763.




\bibitem{Simon} J. Simon: Nonhomogeneous viscous incompressible fluids: existence of velocity, density, and
pressure, {\em SIAM J. Math. Anal.}, {\bf 21} (1990),
 1093--1117.


\bibitem{UY} M.~ R. Ukhovskii, and V. I. Iudovich,  Axially symmetric flows of ideal and viscous fluids
filling the whole space, {\it J. Appl. Math. Mech.}, {\bf 32} (1968)
52-61.

\bibitem{ZZ} P.
Zhang and T.  Zhang,  Global axisymmetric solutions to
three-dimensional Navier-Stokes system, {\it  Int. Math. Res. Not.},
IMRN2014, no. 3, 610-642.
\end{thebibliography}
\end{document}